\documentclass[14pt]{article}
\usepackage{mathpazo}
\usepackage{amsmath}
\usepackage{amsfonts}
\usepackage{array,color}
\usepackage{abstract}
\usepackage[bottom]{footmisc}
\usepackage[left=1.15in, right=1.15in, bottom=0.60in, includefoot]{geometry}
\usepackage{graphicx}
\usepackage{indentfirst}
\usepackage{amsthm}
\usepackage{mathrsfs}
\usepackage{makeidx}
\usepackage{url}
\usepackage{dsfont}
\usepackage{latexsym}
\usepackage{amsmath}
\usepackage{amssymb}
\usepackage{color}
\usepackage{hyperref}
\usepackage{tikz}
\usepackage{MnSymbol}
\numberwithin{equation}{subsection}

\newcommand{\G}{\Gamma}
\newcommand{\g}{\gamma}
\newcommand{\sg}{\sigma}
\newcommand{\mc}{\mathbb{C}}
\newcommand{\mr}{\mathcal R}
\newcommand{\ca}{\curvearrowright}

\newcommand{\emm}{\mathcal{M}}

\newcommand{\enn}{\mathcal N}
\newcommand{\euu}{\mathcal{U}}

\newcommand{\El}{\mathcal{L}}
\newcommand{\Sg}{\Sigma}
\newcommand{\rar}{\rightarrow}
\newcommand{\La}{\Lambda}
\newcommand{\lam}{\lambda}
\newcommand{\bten}{\bar\otimes}
\newcommand{\De}{\Delta}
\newcommand{\tp}{\bar{\otimes}}
\newcommand{\Aa}{\mathcal A}
\newcommand{\Bee}{\mathcal B}
\newcommand{\cp}{\mathcal P}
\newcommand{\tilm}{\tilde{\mathcal M}}

\begin{document}
	\newtheorem{Lemma}{Lemma}
	\theoremstyle{plain}
	\newtheorem{theorem}{Theorem~}[section]
	\newtheorem{main}{Main Theorem~}
	\newtheorem{lemma}[theorem]{Lemma~}
	\newtheorem{assumption}[theorem]{Assumption~}
	\newtheorem{proposition}[theorem]{Proposition~}
	\newtheorem{corollary}[theorem]{Corollary~}
	\newtheorem{definition}[theorem]{Definition~}
	\newtheorem{defi}[theorem]{Definition~}
	\newtheorem{notation}[theorem]{Notation~}
	\newtheorem{example}[theorem]{Example~}
	\newtheorem*{remark}{Remark~}
	\newtheorem*{cor}{Corollary~}
	\newtheorem*{question}{Question}
	\newtheorem*{claim}{Claim}
	\newtheorem*{conjecture}{Conjecture~}
	\newtheorem*{fact}{Fact~}
	\newtheorem*{thma}{Theorem A}
	\newtheorem*{thmb}{Theorem B}
	\newtheorem*{thmc}{Theorem C}
	\newtheorem*{thmd}{Theorem D}
	\newtheorem*{thme}{Theorem E}
	\renewcommand{\proofname}{\bf Proof}
	\newcommand{\email}{Email: }

\title{Semidirect product rigidity of group von Neumann algebras arising from class $\mathscr{S}$, inductive limits and fundamental group}
\author{Sayan Das and Krishnendu Khan}
\date{}
\maketitle
\begin{abstract}
In this article we study property (T) groups arising from Rips construction in geometric group theory in the spirit of \cite{CDK19} and certain inductive limit groups from this class. Using interplay between Popa's deformation/rigidity and methods in geometric group theory we are able to extend the class of groups considered in \cite{CDK19} that remembers semidirect product features while passing to the group von Neumann algebras. Combining these results with the method developed in \cite{CDHK20} we are able to produce more examples of property (T) group factors with trivial fundamental group. The inductive limit groups do not have property (T) and provides examples of more factors with trivial fundamental group. We are also able to show Cartan rigidity for these groups.

\end{abstract}	
\thispagestyle{empty}

	\section{Introduction }
	A natural association of a von Neumann algebra, denoted by $\El(\G)$, to every countable discrete group $\G$ is due to Murray and von Neumann in \cite{MvN37,MvN43}. The group von Neumann algebra $\El(G)$ is defined as the bicommutant of the left regular representation of $\G$ inside the algebra of all bounded linear operators on the Hilbert space $\ell^2\G$. A central theme of study in von Neumann algebras is: what algebraic/geometric properties of group $\G$ are remembered by its group von Neumann algebra $\El(\G)$? Reconstruction of algebraic/geometric properties of groups from its groups von Neumann algebra is extremely difficult as group von Neumann algebras tend to forget quite a lot of group theoretic properties such as rank, torsion etc. For example, $\mathbb  S_{\infty}$ and $\mathbb Z\wr \mathbb Z$ give rise to same group von Neumann algebras. This example is a very special case of Connes' seminal work in \cite{Co76} which states that any two icc amenable group give rise to isomorphic group Von Neumann algebras. Hence any properties of amenable groups are not remembered by the corresponding von Neumann algebras other than amenability.
	
	However, when $\G$ is non-amenable the situation is far more complex and an outstanding progress has been achieved through Popa's deformation/rigidity theory \cite{Po07,Va10,Io12,Io18icm}. Using this new framework it was shown that various properties (algebraic/geometric) of groups and their representation can be completely recovered from their von Neumann algebras, \cite{OP03,OP07,IPV10,BV12,CdSS15,DHI16,CI17,CU18,CDK19,CDHK20,CD-AD20,CD-AD21,CIOS21}. The first success in this direction is due to Ioana, Popa and Vaes in \cite{IPV10} where the authors discovered the first examples of groups that are completely remembered by its group von Neumann algebras, i.e. ``\textit{W$^*$-superrigid groups}''. Since then there have been a huge development in this direction made in \cite{BV12,B13,CI17,CD-AD20,CD-AD21,CIOS21}.  
	
	In order to understand $W^*$-superrigidity, one first need to look at certain aspects of a group $\G$ that are preserved by the group von Neumann algebra $\El(\G)$. In particular we need to identify algebraic properties of a group $\G$ such that we can recover them for any mystery group $\La$ whenever $\El(\G)\cong\El(\La)$. For instance, the examples covered in \cite{IPV10,BV12,B13} and in \cite{CI17} the first step was to reconstruct the generalized wreath product for the unknown group and the amalgam structure respectively. In \cite{CDK19} the novelty was to reconstruct the semidirect product structure.  
	Hence identifying algebraic/geometric features of groups that passes through the von Neumann algebraic structure remain interesting in its own way. Other structural rigidity results were obtain in \cite{CdSS15,CI17,DHI16,CU18}.
	
	A striking conjecture of Connes predicts that all icc property (T) groups are $W^*$-superrigid. Despite having much progress in the direction of $W^*$-rigidity, no example of a property (T) $W^*$-superrigid group is known till date. The first evidence towards this conjecture was found in \cite{CH89}, where it was shown that for different values of $n$, the uniform lattices in $Sp(n,1)$ have non isomorphic von Neumann algebras. It was later generalized in \cite{OP03} in terms of products. The known structural results about property (T) groups beside the aforementioned works are due to \cite{CdSS15,CU18}. A recent progress in that direction was made in \cite{CDK19}, where it was shown that for certain classes of property (T) groups, arising from some groups theoretic Rips constructions, it is possible to reconstruct the semidirect product feature. A striking progress has been made very recently in \cite{CIOS21} by I. Chifan, A. Ioana, D. Osin and B. sun, where the authors have constructed the first examples of an uncountable family of property (T), W$^*$-superrigid groups. 
	
	In this paper we made progress on this direction by extending the semidirect product rigidity results considered in \cite{CDK19}for a larger classes of groups. 
	It was shown in \cite{BO06} that for any finitely generated group $Q$, one can find a property (T) group $N$ such that $Q\hookrightarrow Out(N)$ as a finite index subgroup and the semidirect product group $N\rtimes_{\sg} Q$ with respect to the canonical action $Q\ca^{\sg}N$ is hyperbolic relative to $Q$. For torsion free $Q$ one can choose $N$ to be torsion free as well and if $Q$ has property (T) then $N\rtimes_{\sg}Q$ also has property (T). With all the assumption we denote the classes of groups $N\rtimes_{\sg}Q$ by $Rips_T(Q)$. These classes of groups were first considered and structural rigidity of the associated group von Neumann algebras of these groups has been studied extensively in \cite{CDK19}. 
	
	The first main result of this paper concerns canonical fiber products of groups in $Rips_T(Q)$. In specific, for any $m\geq 2$, consider groups $N_i\rtimes_{\sg_i}Q\in Rips(Q)$  for all $i\in\overline{1,m}$ and the canonical fiber product $\G=(N_1\times\cdots\times N_m)\rtimes_{\sg} Q$ where the action $\sg$ is the diagonal action of $Q$ on $N_1\times\cdots\times N_m$. We show that the semidirect product feature of the group $\G$ survives the passage to the group von Neumann algebraic regime. More precisely we have the following:

	\begin{thma}[Theorem~\ref{semidirectprodreconstruction}] Let $Q=Q_1\times Q_2\times \cdots\times Q_n$, where $n\geq 2$, $Q_i$ are icc, torsion free, biexact, property (T), weakly amenable, residually finite groups. For $i=1,2,\cdots,m$ let $\G=(N_1\times N_2\times\cdots\times N_m)\rtimes_{\sigma} Q$, the semidirect product associated with the diagonal action $\sigma=\sigma_1 \times \sigma_2\times \cdots\times\sigma_m :Q\ca N_1\times N_2\times\cdots\times N_m $. Denote by $\emm=\El(\G)$ be the corresponding II$_1$ factor. Assume that $\Lambda$ is any arbitrary group and $\Theta: \El(\Gamma)\rar \El(\Lambda)$ is any $\ast$-isomorphism. Then there exist groups action by automorphisms $H\ca^{\tau_i} K_i$  such that $\La= (K_1\times K_2\times \cdots\times K_m) \rtimes_{\tau} H$ where $\tau=\tau_1\times \tau_2\times\cdots\times \tau_m: H\ca K_1\times K_2\times\cdots\times K_m$ is the diagonal action. Moreover one can find a multiplicative character $\eta:Q\rar \mathbb T$, a group isomorphism $\delta: Q\rar H$, a unitary $w\in \El(\Lambda)$, and $\ast$-isomorphisms $\Theta_i: \El(N_i)\rar \El(K_i)$ such that for all $x_i \in L(N_i)$ and $g\in Q$ we have 
		\begin{equation}
			\Theta((x_1\otimes x_2\otimes\cdots\otimes x_m) u_g)= \eta(g) w ((\Theta_1(x_1)\otimes \Theta_2(x_2)\otimes\cdots\otimes(\Theta_m(x_m))) v_{\delta(g)})w^*.
		\end{equation}      
		Here $\{u_g \,|\,g\in Q\}$ and $\{v_h\,|\, h\in H\}$ are the canonical unitaries implementing the actions of the group $Q$ on $ \El(N_1)\bar\otimes \El(N_2)\bar{\otimes}\cdots\bar{\otimes}\El(N_m)$ and $H$ on $ \El(K_1)\bar\otimes \El(K_2)\bar{\otimes}\cdots\bar{\otimes}\El(K_m)$, respectively.
	\end{thma}
   We denote the class of groups that is considered in the above theorem, by \textbf{class} $\mathscr S_{m,n}$. Note that the \textbf{class}$\mathscr S_{2,2}$ is precisely \textbf{class}$\mathscr S$ as defined in \cite{CDK19}.
   
   Our next result concerns about the inductive limit of groups from \textbf{class}$\mathscr S_{m,n}$. Specifically, consider $\G_m:=(N_1\times\cdots\times N_m)\rtimes Q\in \mathscr S_{m,n}$ and consider the limit group $\G =\underset{\rightarrow}\lim \G_m$. We denote this class of groups by \textbf{class} $\mathscr S_{\infty,n}$. For fairly large family of $Q$, we showed that the semidirect product feature of this limit group can be also recovered and more interestingly the acting group can be completely determined. The precise theorem is following:   
	\begin{thmb}[Theorem~\ref{semidirect-S_inftystructure}]
	Let $\G=N\rtimes Q\in\mathscr{S}_{\infty,n}$, $\La$ be any group and $\Theta:\El(\G)\rightarrow \El(\La)$ be a $\ast$-isomorphism. Then the group $\La$ admits a semidirect product decomposition  $\La=\Sg\rtimes \Phi $ satisfying the following properties: There is a group isomorphism $\delta: Q\rightarrow \Phi$, a character $\zeta:Q\rightarrow \mathbb T$, a $\ast$-homomorphism $\Theta_0:\El(N)\rightarrow \El(\Sg)$ and a unitary $w\in \El(\La)$ such that for every $x\in\El(N)$ and $\g\in Q$ we have 
	$$\Theta(xu_{\g})=\zeta(\g)w\Theta_0(x)v_{\delta(\g)}w^{\ast} .$$
	Where $\{ u_{\g}|\g\in Q\}$ and $\{v_{\lambda}|\lambda\in \Phi\}$ are the canonical unitaries of of $\El(Q)$ and $\El(\Phi)$ respectively.
\end{thmb}

    Amplification of a type II$_1$ factor $\emm$ by any positive real number $t>0$ was introduced by Murray and von Neumann in \cite{MvN43}. This is a II$_1$ factor denoted by $\emm^t$ and called the $t$-amplification of $\emm$. When $ t\leq 1$ this is the isomorphism class of $p\emm p$ for a projection $p\in \emm$ of trace $\tau(p)=t$ and when $1<t$, it is the isomorphism class of $p (M_n(\mathbb C) \otimes \emm)p$ for an integer $n$ with $t/n \leq 1$ and a projection $p\in M_n(\mathbb C) \otimes \emm$ of trace $(Tr_n\otimes \tau)(p)=t/n$. One can see that up to isomorphism the $\emm^t$ does not depend on $n$ or $p$ but only on the value of $t$. The fundamental group of a II$_1$ factor $\emm$ is defined by $\mathcal F(\emm):=\{t>0|\emm^t\cong \emm\}$ as a multiplicative subgroup of $\mathbb R_{+}$. As the fundamental group is an isomorphism invariant of the factor, its study is of central importance to the theory of von Neumann algebras. In \cite{MvN43} Murray and von Neumann were able to show  that the fundamental group of the hyperfinite $\rm II_1$ factor $\mr$ satisfies $\mathcal F(\mr)= \mathbb R_{+}$. This also implies that $\mathcal F = \mathbb R_{+} $ for all McDuff factors $\emm$.
    A breakthrough in this direction emerged from Connes' discovery in \cite{Co80} that the fundamental group of a  group factor $\mathcal F(\El(G))$ reflects rigidity aspects of the underlying group $G$, being countable whenever $G$ has property (T) of Kazdhan \cite{Kaz67}. This motivated him to formulate his famous Rigidity Conjecture in \cite{Co82} along with other problems on computing symmetries of property (T) factors---that were highlighted in subsequent articles by mathematicians \cite[Problem 2, page 551]{Co94}, \cite[Problems 8-9]{Jo00} and \cite[page 9]{Po13}. 
    
    However, despite  these impressive achievements, significantly less is known about the fundamental groups of property (T) factors. A significant progress towards the fundamental group of property (T) factors was first made in \cite{CDHK20}, where two classes of property (T) group factors were shown to have trivial fundamental group. Later it was observed in \cite{D20} that products of these two classes of groups also have trivial fundamental group. An uncountable family of wreath-like product groups with property (T) is constructed very recently in \cite{CIOS21} whose fundamental group is trivial. Beside the aforementioned semidirect product rigidity results we also made progress in this direction by showing that the group factors arising from \textbf{class} $\mathscr{S}_{m,n}$ have trivial fundamental group. This gives more examples of property (T) group factors with trivial fundamental group. We can also show that group factors coming from \textbf{class} $\mathscr S_{\infty,n}$ also have trivial fundamental group. These groups do not have property (T), however, they arise as inductive limits of factors. This result is somehow counter intuitive as for inductive limit situation one would expect the fundamental group to be $\mathbb R_{+}$ (sec section for detail explanation).

\begin{thmc}[Corollary~\ref{trivialfundamentalgroup_S_m,n} and Corollary~\ref{trivialfundamentalgroupS_infty,n}]
	For $\G\in \mathscr{S}_{m,n}\cup\mathscr{S}_{\infty,n}$, we have $\mathcal F(\El(\G))=\{1\}$.
\end{thmc}

   In section~\ref{uniqueprimeness}, we provide another family of property (T) group factors with trivial fundamental group by showing unique prime factorization for products of groups arising from class $\mathscr S_{m,n}$ and class $\mathscr V$.
   
\begin{thmd}[Theorem~\ref{upffrom_s_mn}]
	For $\G_1,\G_2\in(\underset{m,n\geq 2}{\bigcup}\mathscr S_{m,n})\cup \mathscr V$, we have $\mathcal{F}(\G_1\times \G_2)=\{1\}$.
\end{thmd}   

    Recall that a group $\G$ is said to $\mathcal C$-rigid (Cartan-rigid) in the sense of \cite{PV11} if the $\rm II_1$ factor $L^{\infty}(X)\rtimes \G$ has a unique Cartan subalgebra\footnote{A subalgebra $\mathcal A\subseteq \emm $ of a von Neumann algebra $\emm$ is a Cartan subalgebra if it is a maximal abelian von Neumann subalgebra of $\emm$ and its normalizer $\mathcal N_{\emm}(\mathcal A):=\{u\in\mathcal U(\emm):u\mathcal A u^*=\mathcal A \}$ generates $\emm$ as a von Neumann algebra} up to conjugacy for any ergodic, free, p.m.p. action $\G\ca X$. In the last section we obtain Cartan rigidity results for these group factors and the crossed-product von Neumann algebras associated to their free, ergodic, p.m.p. actions. Theorem~E provides new examples of $\mathcal C$-rigid groups, adding to the classes already discovered in \cite{PV11,PV12,Io12,CIK13,CDK19,CK-E21}.
    
    \begin{thme}[Corollary~\ref{cor:nocartan}] 
    Let $\G$ be either a product of groups that are hyperbolic relative to a family of residually finite subgroups, or products of groups from \textbf{class} $\mathscr S_{m,n}$. Then $\El(\G)$ has no Cartan subalgebras and $\G$ is $\mathcal C$-rigid.	
    \end{thme}

	
	\section{Preliminaries}
	
	\subsection{Notations and Terminology}

	We denote by $\mathbb N$ and $\mathbb Z$ the set of natural numbers and the integers, respectively. For any $k\in \mathbb N$ we denote by $\overline{1,k}$ the integers $\{1,2,...,k\}$. 
\vskip 0.05in	
	All von Neumann algebras in this document will be denoted by calligraphic letters e.g.\ $\mathcal A$, $\mathcal B$, $\mathcal M$, $\mathcal N$, etc. Given a von Neumann algebra $\mathcal M$ we will denote by $\mathscr U(\mathcal M)$ its unitary group, by $\mathscr P(\mathcal M)$ the set of all its nonzero projections, and by $\mathscr Z(\mathcal M)$ its center. We also denote by $(\mathcal M)_1$ its unit ball. All algebras inclusions $\mathcal N\subseteq \mathcal M$ are assumed unital unless otherwise specified. Given an inclusion $\mathcal N\subseteq \mathcal M$ of von Neumann algebras we denote by $\mathcal N'\cap \mathcal M$ the relative commutant of $\mathcal N$ in $\mathcal M$, i.e.\ the subalgebra of all $x\in \mathcal M$ such that $xy=yx$ for all $y\in \mathcal N$. We also consider the one-sided quasinormalizer $\mathscr {QN}^{(1)}_{\mathcal M}(\mathcal N)$ (the semigroup of all $x\in\mathcal M$ for which there exist $x_1,x_2,...,x_n \in \mathcal M$ such that $\mathcal N x\subseteq \sum_i x_i \mathcal N$) and the quasinormalizer $\mathscr {QN}_{\mathcal M}(\mathcal N)$ (the set of all $x\in\mathcal M$ for which there exist $x_1,x_2,...,x_n \in \mathcal M$ such that $\mathcal N x\subseteq \sum_i x_i \mathcal N$ and $x\mathcal N \subseteq \sum_i  \mathcal N x_i$) and we notice that $\mathcal N\subseteq \mathscr{N}_{\mathcal M}(\mathcal N)\subseteq \mathscr {QN}^{}_{\mathcal M}(\mathcal N)\subseteq \mathscr {QN}^{(1)}_{\mathcal M}(\mathcal N)$.
\vskip 0.05in
All von Neumann algebras $\emm$ considered in this article will be tracial, i.e.\ endowed with a unital, faithful, normal linear functional $\tau:M\rightarrow \mathbb C$  satisfying $\tau(xy)=\tau(yx)$ for all $x,y\in \emm$. This induces a norm on $\emm$ by the formula $\|x\|_2=\tau(x^*x)^{1/2}$ for all $x\in \emm$. The $\|\cdot\|_2$-completion of $\emm$ will be denoted by $L^2(\emm)$.  For any von Neumann subalgebra $\mathcal N\subseteq \mathcal M$ we denote by $E_{\mathcal N}:\mathcal M\rightarrow \mathcal N$ the $\tau$-preserving condition expectation onto $\mathcal N$. 
\vskip 0.05in
For a countable group $G$ we denote by $\{ u_g | g\in G \} \in \mathscr U(\ell^2G)$ its left regular representation given by $u_g(\delta_h ) = \delta_{gh}$, where $\delta_h:G\rightarrow \mathbb C$ is the Dirac function at $\{h\}$. The weak operatorial closure of the linear span of $\{ u_g | g\in G \}$ in $\mathscr B(\ell^2 G)$ is the so called group von Neumann algebra and will be denoted by $\El(G)$. $\El(G)$ is a II$_1$ factor precisely when $G$ has infinite non-trivial conjugacy classes (icc). If $\mathcal M$ is a tracial von Neumann algebra and $G \ca^\sigma \mathcal M$ is a trace preserving action we denote by $\mathcal M \rtimes_\sigma G$ the corresponding cross product von Neumann algebra \cite{MvN37}. For any subset $K\subseteq G$ we denote by $P_{\mathcal M K}$  the orthogonal projection from the Hilbert space $L^2(\mathcal M \rtimes G)$ onto the closed linear span of $\{x u_g \,|\, x\in \mathcal M, g\in K\}$. When $\mathcal M$ is trivial we will denote this simply by $P_K$.  
\vskip 0.05in
Given a subgroup  $H \leqslant G$ we denote by $C_G(H)$ the centralizer of $H$ in $G$ and by $N_G(H)$ the normalizer of $H$ in $G$. Also we will denote by $QN^{(1)}_G(H)$ the  one-sided quasinormalizer of $H$ in $G$; this is the semigroup of all $g\in G$ for which there exist a finite set $F \subseteq G$ such that $Hg\subseteq F H$. Similarly we denote by $QN_G(H)$ the  quasinormalizer (or commensurator) of $H$ in $G$, i.e.\ the subgroup of all $g\in G$ for which there is a finite set $F \subseteq G$ such that $Hg\subseteq FH$ and $gH\subseteq HF$. We canonically have $H C_G(H)\leqslant N_G(H)\leqslant QN_G(H)\subseteq QN^{(1)}_G(H)$. We often consider the virtual centralizer of $H$ in $G$, i.e. $vC_G(H)=\{g\in G \,|\, |g^{H}|<\infty\} $. Notice $vC_G(H)$ is a subgroup of $G$ that is normalized by $H$. When $H=G$, the virtual centralizer is nothing else but the FC-radical of $G$. Also one can easily see from definitions that $HvC_G(H)\leqslant QN_G(H)$. For a subgroup $H\leqslant G$ we denote by $\llangle H\rrangle$ the normal closure of $H$ in $G$.

Finally, for any groups $G$ and $N$ and an action $G\curvearrowright^{\sg} N$ we denote by $N\rtimes_\sigma G$ the corresponding semidirect product group.

	\subsection{Popa's Intertwining Techniques} Sorin Popa has introduced  in \cite [Theorem 2.1 and Corollary 2.3]{Po03} a beautiful analytic criterion for identifying intertwiners between arbitrary subalgebras of tracial von Neumann algebras. Now this is known in the literature  as \emph{Popa's intertwining-by-bimodules technique} and has played a key role in the classification of von Neumann algebras.

\begin {theorem}\cite{Po03} \label{corner} Let $(\mathcal M,\tau)$ be a separable tracial von Neumann algebra and let $\mathcal P, \mathcal Q\subseteq \mathcal M$ be (not necessarily unital) von Neumann subalgebras. 
Then the following are equivalent:
\begin{enumerate}
\item There exist $ p\in  \mathscr P(\mathcal P), q\in  \mathscr P(\mathcal Q)$, a $\ast$-homomorphism $\theta:p \mathcal P p\rightarrow q\mathcal Q q$  and a partial isometry $0\neq v\in q \mathcal M p$ such that $\theta(x)v=vx$, for all $x\in p \mathcal P p$.
\item For any group $\mathcal G\subset \mathscr U(\mathcal P)$ such that $\mathcal G''= \mathcal P$ there is no sequence $(u_n)_n\subset \mathcal G$ satisfying $\|E_{ \mathcal Q}(xu_ny)\|_2\rightarrow 0$, for all $x,y\in \mathcal  M$.
\item There exist finitely many $x_i, y_i \in \mathcal M$ and $C>0$ such that  $\sum_i\|E_{ \mathcal Q}(x_i u y_i)\|^2_2\geq C$ for all $u\in \mathcal U(\mathcal P)$.
\end{enumerate}
\end{theorem} 
\vskip 0.02in
\noindent If one of the three equivalent conditions from Theorem \ref{corner} holds then we say that \emph{ a corner of $\mathcal P$ embeds into $\mathcal Q$ inside $\mathcal M$}, and write $\mathcal P\prec_{\mathcal M}\mathcal Q$. If we moreover have that $\mathcal P p'\prec_{\mathcal M}\mathcal Q$, for any projection  $0\neq p'\in \mathcal P'\cap 1_{\mathcal P} \mathcal M 1_{\mathcal P}$ (equivalently, for any projection $0\neq p'\in\mathscr Z(\mathcal P'\cap 1_{\mathcal P}  \mathcal M 1_{P})$), then we write $\mathcal P\prec_{\mathcal M}^{s}\mathcal Q$.
\vskip 0.02in
	
For further use we record the following result which controls the intertwiners in algebars arsing form malnormal subgroups. Its proof is essentially contained in \cite[Theorem 3.1]{Po03} so it will be left to the reader. 	  
\begin{lemma}[\cite{Po03}]\label{malnormalcontrol}Assume that $H\leqslant G$ be an almost malnormal subgroup and let $G \ca \mathcal N$ be a trace preserving action on a finite von Neumann algebra $\mathcal N $. Let $\mathcal P \subseteq \mathcal N \rtimes H$ be a von Neumann algebra such that $\mathcal P\nprec_{\mathcal N\rtimes H}  N$. Then for every elements $x,x_1,x_2,...,x_l \in \mathcal N\rtimes G$ satisfying $\mathcal P x\subseteq \sum^l_{i=1} x_i \mathcal P$ we must have that $x\in \mathcal N\rtimes H$. 
\end{lemma}

We now state the following lemma that is a variation of \cite[Lemma~9.2(1)]{Io11}.

\begin{lemma}\label{io11gen}
	Let $\G_1\times \G_2= \La$ be an countable groups such that $ \La\times \La=\G_1\times\G_2\times \G_1\times \G_2 $. Let $\Delta:\emm\rightarrow\emm\bar{\otimes}\emm$ be the comultiplication map given by $\Delta(u_g)=u_g\otimes u_g$ where $\emm=\El(\La)$. If $\Delta(\El(Q))\prec_{ \mathcal K \bten \El(\G_2)}\mathcal{K}\bten 1$, then $\El(Q)\prec_{\emm} \mathbb C 1$, where $\mathcal K=\El(\G_1\times\G_2\times\G_1)$.
\end{lemma}
\begin{proof}
	Let $\iota:Q\rightarrow \G_i\times \G_2$ be the embedding of groups and $\pi_i:\La\rightarrow \G_i$ be corresponding projection onto $\G_i$. We now view $Q$ inside $\G_1\times\G_2\times\G_1$ as the embedding $\beta:Q\rightarrow \G_1\times\G_2\times\G_1$ where $\beta(g):=(\pi_1\circ\iota(g),\pi_2\circ\iota(g),\pi_1\circ\iota(g))$. Assume that $\El(Q)\nprec_{\mathcal K} \mathbb C 1$. So we can find a sequence of unitaries $\{u_n\}_{n\geq 1} \subset \euu(\El Q)$ such that $\|E_{\mathbb{C}1}(au_nb)\|_2\rightarrow 0$ for all $a,b\in \mathcal K$. We now claim that
	\begin{align}
		\|E_{\mathcal K\bten 1}(x\Delta(u_n)y)\|_2\rightarrow 0\ for\ all\ x,y\in \emm\bten \emm
	\end{align}
Since $E_{\mathcal K\bten1}$ is $\mathcal{K}\bten 1$-bimodular, we can assume that $x=1\otimes au_g,y=1\otimes bu_h$ for some $a,b\in\mathbb{C}$ and $g,h\in \G_2$. For every $n\geq 1$ we have:
$$x\Delta(u_n)y=\sum_{\lam\in\G_1}\tau(1\otimes au_g(\Delta(u_n))1\otimes bu_h u^*_{\lam})u_{\lam}=\sum_{\lam\in\G_2}\tau(1\otimes au_g(w_n\otimes v_n)1\otimes bu_h (s^*_{\lam}\otimes t^*_{\lam}))s_{\lam}\otimes t_{\lam}$$ 
$$=\sum_{\lam\in\G_1}\tau(w_ns^*_{\lam})s_{\lam}\otimes au_gv_nbu_h$$
Thus we get that $\|E_{\mathcal K\bten 1}(x\Delta(u_n)y)\|_2\rightarrow 0$, this implies that $\Delta(\El(Q))\nprec_{ \mathcal K \bten \El(\G_2)}\mathcal{K}\bten 1$. Hence we have $\El(Q)\prec_{ \mathcal K}\mathbb{C}$. Now by \cite[Theorem~2.4]{CDK19} we get that $\El(Q)\prec_{ \emm}\mathbb{C}$.
\end{proof}


\subsection{Class $\mathscr{S}_{m,n}$}

Belegradek and Osin showed in \cite[Theorem 1.1]{BO06} that for every finitely generated group $Q$ one can find a property (T) group $N$ such that $Q$ embeds into ${\rm Out}(N)$ as a finite index subgroup. This canonically gives rise to an action $Q \curvearrowright^\rho N$ by automorphisms such that the corresponding semidirect product group $N\rtimes_\rho Q$ is hyperbolic relative to $\{Q\}$. Throughout this document the semidirect products $N\rtimes_\rho Q$ will be termed Rips construction groups. When $Q$ is torsion free then one can pick $N$ to be torsion free as well and hence both $N$ and $N\rtimes_\rho Q$ are icc groups. Also when $Q$ has property (T) then $N\rtimes_\rho Q$ has property (T). Under all these assumptions we will denote by $\mathcal Rips_T(Q)$ the class of these Rips construction groups $N\rtimes_\rho Q$.

In \cite[Sections 3,5]{CDK19} the authors introduced a class of property (T) groups based on fiber products of Rips construction groups and have proved several rigidity results for the corresponding von Neumann algebras, \cite[Theorem A]{CDK19}. In \cite{CDHK20} the authors showed that the group von Neumann algebras corresponding to the classes of groups considered in \cite{CDK19} have trivial fundamental group. Next we briefly recall this construction and also define the class $\mathscr{S}_{m,n}$ which generalizes the class $\mathscr{S}$.

For $n\geq 2$, consider any product group $Q= Q_1\times Q_2\times \cdots \times Q_n$, where $Q_i$ are any nontrivial, bi-exact, weakly amenable, property (T), residually finite, torsion free, icc groups. Then for every $i=1,2,\cdots,m$  consider a Rips construction  $G_i = N_i \rtimes_{\rho_i} Q\in \mathcal Rips_T(Q)$, let $N=N_1\times N_2\times\cdots\times N_m$  and denote by $G= N\rtimes_\sigma Q$ the canonical semidirect product which arises from the diagonal action  $\sigma=\rho_1\times \rho_2\times\cdots\times\rho_m: Q\rar {\rm Aut}(N)$, i.e. $\sigma_g (h_1,h_2,\cdots,h_m)=( (\rho_1)_g(h_1), (\rho_2)_g(h_2),\cdots,(\rho_m)_g(h_m))$ for all $(h_1,h_2,\cdots,h_m)\in N$. Throughout this article the category of all these  semidirect products $G$ will denoted by {\bf Class $\mathscr S_{m,n}$}. Note that {\bf Class $\mathscr S_{2,2}$} is precisely the {\bf Class $\mathscr S$} as defined in \cite{CDK19}.

Concrete examples of semidirect product groups in class  $\mathscr S_{m,n}$ can be obtained if the initial groups $Q_i$ are any uniform lattices in $Sp(n,1)$ when $n\geq 2$. Indeed one can see that required conditions on $Q_i$'s follow from \cite{Oz03,CH89}.

For further reference and algebraic properties of groups in class $\mathscr S$, the reader may consult \cite[Sections 3,4,5]{CDK19} and the references within. 

\subsection{Class $\mathscr{S}_{\infty,n}$}

Consider the product group $Q=Q_1\times\cdots\times Q_n$, where $n\geq 2$ and $Q_i$ are any nontrivial, bi-exact, weakly amenable, property (T), residually finite, torsion free, icc groups. Now consider the groups $N_i\rtimes Q\in Rips_T(Q)$ for $i\in\mathbb N$. Let $N=\underset{i\in\mathbb N}{\oplus} N_i$ and consider the group $\G=N\rtimes_{\sg} Q$, where the action $\sg$ is the diagonal action. The category of all these groups $\G$ will be denoted by \textbf{Class} $\mathscr{S}_{\infty,n}$.
\vskip 0.1in 
Note that for $\G\in\mathscr{S}_{\infty,n}$, $\El(\G)$ does not have property (T) and can be written as increasing union of irreducible property (T) group factors. Let $\G_k:=(N_1\times\cdots\times N_k)\rtimes Q$, then observe that $\G_k\in\mathscr{S}_{k,n}$ has property (T) and $\El(G)=(\underset{k\in \mathbb N,k\geq 2}{\bigcup} \El(\G_k))''$.


\section{Semidirect product rigidity for group von Neumann algebras associated to class $\mathscr{S}_{m,n}$}

A celebrated conjecture of Connes predicts that all icc property (T) groups are $W^*$-superrigid. Recently, in a beautiful paper \cite{CIOS21}, the authors constructed a class of property (T) W$^*$-superrigid group. Unfortunately, not much is known at this time except this class of wreath like product groups (contains uncountably many property (T) groups). Moreover, in the current literature there is an almost complete lack of examples of algebraic features occurring in a property (T) group that are recognizable at the von Neumann algebraic level. For instance, besides the preservance of the Cowling-Haagerup constant \cite{CH89}, the amenablity of normalizers of infinite amenable subgroups in hyperbolic property (T) groups from \cite[Theorem 1]{Oz03}, the product rigidity for hyperbolic property (T) groups from \cite[Theorem A]{CdSS15}, and wreath like product structure from \cite[Corollary~1.6, Theorem~1.7]{CIOS21} very little is known. Therefore in order to successfully construct property (T) $W^*$-superrigid groups via a strategy similar to \cite{IPV10,CI17} we believe it is imperative to first identify a comprehensive list of algebraic features of property (T) groups that survive the von Neumann algebraic structure. Any success in this direction will potentially hint to what group theoretic methods to pursue in order to address Connes' conjecture.   
\vskip 0.07in
In this section we investigate groups arising from class $\mathscr S_{m,n}$. Notice that since property (T) is closed under extensions \cite[Section 1.7]{BdlHV00} it follows that $\G\in \mathscr S_{m,n}$  has property (T). Then for a fairly large family of groups $Q$ we show that the semidirect product feature of $\G$ is an algebraic property completely recoverable from the von Neumann algebraic regime. In addition, we also have a complete reconstruction of the acting group $Q$. The precise statement is the following

\begin{theorem}\label{semidirectprodreconstruction} Let $Q=Q_1\times Q_2\times \cdots\times Q_n$, where $n\geq 2$, $Q_i$ are icc, torsion free, biexact, property (T), weakly amenable, residually finite groups. For $i=1,2,\cdots,m$ with $m\geq 2$, let $N_i\rtimes_{\sigma_i} Q\in \mathcal Rip_T(Q)$ and denote by $\G=(N_1\times N_2\times\cdots\times N_m)\rtimes_{\sigma} Q$ the semidirect product associated with the diagonal action $\sigma=\sigma_1 \times \sigma_2\times \cdots\times\sigma_m :Q\ca N_1\times N_2\times\cdots\times N_m $. Denote by $\emm=\El(\G)$ be the corresponding II$_1$ factor. Assume that $\Lambda$ is any arbitrary group and $\Theta: \El(\Gamma)\rar \El(\Lambda)$ is any $\ast$-isomorphism. Then there exist groups action by automorphisms $H\ca^{\tau_i} K_i$  such that $\La= (K_1\times K_2\times \cdots\times K_m) \rtimes_{\tau} H$ where $\tau=\tau_1\times \tau_2\times\cdots\times \tau_m: H\ca K_1\times K_2\times\cdots\times K_m$ is the diagonal action. Moreover one can find a multiplicative character $\eta:Q\rar \mathbb T$, a group isomorphism $\delta: Q\rar H$, a unitary $w\in \El(\Lambda)$, and $\ast$-isomorphisms $\Theta_i: \El(N_i)\rar \El(K_i)$ such that for all $x_i \in L(N_i)$ and $g\in Q$ we have 
	\begin{equation}
	\Theta((x_1\otimes x_2\otimes\cdots\otimes x_m) u_g)= \eta(g) w ((\Theta_1(x_1)\otimes \Theta_2(x_2)\otimes\cdots\otimes(\Theta_m(x_m))) v_{\delta(g)})w^*.
	\end{equation}      
	Here $\{u_g \,|\,g\in Q\}$ and $\{v_h\,|\, h\in H\}$ are the canonical unitaries implementing the actions of $Q \ca \El(N_1)\bar\otimes \El(N_2)\bar{\otimes}\cdots\bar{\otimes}\El(N_m)$ and $H\ca \El(K_1)\bar\otimes \El(K_2)\bar{\otimes}\cdots\bar{\otimes}\El(K_m)$, respectively.
\end{theorem}

From a different perspective our theorem can be also seen as a von Neumann algebraic superrigidity result regarding conjugacy of actions on noncommutative von Neumann algebras. Notice that very little is known in this direction as well, as most of the known superrigidity results concern algebras arising from actions of groups on probability spaces.    
\vskip 0.07in
We continue with a series of preliminary results that are essential to derive the proof of Theorem \ref{semidirectprodreconstruction} at the end of the section. First we recall a location result for commuting diffuse property (T) subalgebras inside a von Neumann algebra arising from products of relative hyperbolic groups from \cite{CDK19}.

\begin{theorem}\cite[Theorem~5.2]{CDK19}\label{controlprodpropt1} For $i=1, ...,n$ let  $H_i<G_i$ be an inclusion of infinite groups such that $H_i$ is residually finite and $G_i$ is hyperbolic relative to $H_i$. Denote by $H=H_1\times ...\times H_n < G_1\times ...\times G_n=G$ the corresponding direct product inclusion. Let $\enn_1,\enn_2 \subseteq \El(G)$ be two commuting von Neumann subalgebras with property (T). Then for every $k\in\overline{1,n}$ there exists $i\in \overline{1,2}$ such that $\enn_i \prec \El(\hat G_k\times H_k)$, where $\hat G_k :=\times _{j\neq k} G_j$.

\end{theorem}

\begin{theorem}\label{controlprodpropt2} Under the same assumptions as in Theorem \ref{controlprodpropt1} for every $k\in \overline{1,n}$ one of the following must hold \begin{enumerate}
\item [1)] there exists $i\in 1,2$ such that $\enn_i \prec_\emm \El(\hat G_k)$; 
\item [2)] $\enn_1\vee \enn_2 \prec_{\emm} \El(\hat G_k \times H_k)$.
\end{enumerate}

\end{theorem}

We now proceed towards proving the main result of this chapter. To simplify the exposition we first introduce a notation that will be used throughout the section.  
\begin{notation}\label{semidirectt} Denote by $Q= Q_1\times Q_2\times \cdots\times Q_n$, where $Q_i$ are infinite, residually finite, biexact, property (T), icc groups. Then consider $\G_i = N_i \rtimes Q\in \mathcal Rip_T(Q)$ and consider the semidirect product $\G= (N_1\times N_2\times\cdots\times N_m)\rtimes_\sigma Q$ arising from the diagonal action  $\sigma=\sigma_1\times \cdots\sigma_m: Q\rar Aut(N_1\times N_2\times\cdots\times N_m)$, i.e. $\sigma_g (n_1,\cdots,n_m)=( (\sigma_1)_g(n_1),\cdots, (\sigma_m)_g(n_m))$ for all $(n_1,\cdots,n_m)\in N_1\times N_2\times\cdots\times N_m$. For further use we observe that $\G$ is the fiber product $\G=\times_Q\G_i $ and thus embeds into $\G_1\times\cdots\times  \G_m$ where $Q$ embeds diagonally into $Q\times\cdots\times  Q$. Over the next proofs when we refer to this copy we will often denote it by $diag(Q)$. Also notice that $\G$ is an icc group with property (T) as it arises from an extension of property (T) groups. 
\end{notation}

\begin{theorem}\label{commutationcontrolincommultiplication} Let $\G$ be a group as in Notation \ref{semidirectt} and assume that $\Lambda$ is a group such that $\El(\Gamma)=\El(\Lambda)=\emm$. Let $\Delta: \emm\rar \emm\bar \otimes \emm$ be the ``comultiplication along $\Lambda$'' i.e. $\Delta (v_\lambda)=v_\lambda \otimes v_\lambda$. Then the following hold:
\begin{enumerate}
\item [3)] for all $j\in \overline{1,m}$ and every proper subset $S\subseteq \{1,\cdots,m\}$ we have either $\Delta(\El( \widehat N_S))\prec_{\emm\bar\otimes \emm} \emm\bar\otimes \El(\widehat N_j)$ or $\Delta(\El(N_S))\prec_{ \emm \bten \emm} \emm\bar\otimes \El(\widehat N_j)$, and
\item [4)] 
\begin{enumerate} \item [a)] for all $j\in\overline{1,m}$ and every proper subset $S\subseteq \{1,\cdots,m\}$ we have either $\Delta (\El(Q_S))\prec_{\emm\bar\otimes \emm} \emm\bar\otimes \El(\widehat N_j)$ or  $\Delta (\El(\widehat Q_S))\prec_{\emm\bar\otimes \emm} \emm\bar\otimes \El(\widehat N_j)$, or
\item [b)]$\Delta(\El(Q))\prec_{\emm\bar\otimes \emm} \emm\bar\otimes \El(Q)$; moreover in this case for every $j\in 1,2$ there is $i\in 1,2$ such that $\Delta (\El(Q_j))\prec_{\emm\bar\otimes \emm} \emm\bar\otimes \El(Q_i)$ 
\end{enumerate}   
\end{enumerate}

\end{theorem}

\begin{proof} Let $\tilde \emm=\El(\G_1\times \G_2\times\cdots\times \G_m)$. Since $\G<\G_1\times \G_2\times\cdots\times \G_m$ we notice the following inclusions $\Delta(\El(N_i)), \Delta
 (\El(\hat N_i))\subset \emm\bar \otimes \emm = \El(\G\times \G)\subset \El(\G_1\times \G_2\times\cdots\times \G_m \times \G_1\times \G_2\times \cdots\times\G_m)$. Since $\G_i$ is hyperbolic relative to $Q$ then using Theorem \ref{controlprodpropt2} we have either 
 \begin{enumerate}
 \item [5)]for every $i\in\overline{1,m}$, $\Delta(\El(N_i))\prec_{\tilde \emm \bten \tilde \emm} \emm\bar\otimes \El(\hat \G_1)$, or  $\Delta(\El(\hat N_i))\prec_{\tilde \emm \bten \tilde \emm} \emm\bar\otimes \El(\hat \G_1)$, or
 \item [6)]$\Delta(\El(N))\prec_{\tilde \emm \bten \tilde \emm} \emm\bten \El(\hat \G_1\times Q)$
\end{enumerate}  
Assume $\De(\El(N_i))\subset \emm\bten \El(\G)$ in $5)$, then by \cite[Lemma~2.3]{CDK19} there is a $h\in \G \times \G $ so that $\De(\El(N_i))\prec_{\tilde \emm \bten \tilde \emm} \El(\G\times \G \cap h (\G \times \hat \G_1) h^{-1}))=\El(\G \times (\G\cap\hat \G_1 ))= \emm\bten \El(N \rtimes diag(Q))\cap (\hat N_1 \rtimes Q\times 1))= \emm\bten \El(\hat N_1)$. Note that since $\Delta(\El(N_i))$ is regular in $\emm \tp \emm$, using \cite[Lemma~2.4]{CDK19}, we get that $\Delta(\El(N_i))\prec_{ \emm \bten \emm} \emm\bar\otimes \El(\hat N_1)$.

If we assume $\De(\El(\hat N_i))\subset \emm\bten \El(\G)$ in $5)$, arguing similarly we get that $\Delta(\El(\hat N_i))\prec_{ \emm \bten \emm} \emm\bar\otimes \El(\hat N_1)$, thereby establishing 3).

Assume 6). we can follow the proof on Theorem of \cite{CDK19} to conclude that $3)$ holds.

\vskip 0.05in 
Next we derive 4). Again we notice that $\De(\El(Q_1))$, $\De(\El(\hat Q_1))\subset \De(\emm)\subset \emm\bten \emm =\El(\G\times \G)\subset \El(\G_1\times \G_2 \times\cdots\times\G_m\times\G_1\times \G_2\times\cdots\times\G_m)$. Using Theorem \ref{controlprodpropt2} we must have that either 
\begin{enumerate}
\item[7)]\label{int5}  $\De(\El(Q_i))\prec_{\tilde \emm\bten \tilde \emm } \emm\bten \El(\hat \G_1)$, or $\De(\El(\hat Q_i))\prec_{\tilde \emm\bten \tilde \emm } \emm\bten \El(\hat \G_1)$, or 
\item [8)]\label{int6} $\De(\El(Q))\prec_{\tilde \emm\bten \tilde \emm} \emm\bten \El(\hat \G_1 \times Q)$.
\end{enumerate}  

Proceeding exactly as in the previous case, and using \cite[Lemma~2.4]{CDK19}, we see that  7) implies $\De(\El(Q_i))\prec_{\emm\bten \emm} \emm\bten \El(\hat N_1)$ or $\De(\El(\hat Q_i))\prec_{\emm\bten \emm} \emm\bten \El(\hat N_1)$ which in turn gives  $4a)$. Also proceeding as in the previous case, and using \cite[Lemma~2.5]{CDK19}, we see that 8) implies 
\begin{equation}\label{int7}\De(\El(diag (Q))\prec_{ \emm\bten \emm} \emm\bten \El(\hat N_1 \rtimes diag(Q)).  
\end{equation}

To show the part $4b)$ we will exploit \eqref{int7}. Notice that there exist nonzero projections $ r\in \De(\El(Q))$, $t\in \emm\bten \El(\hat N_1\rtimes diag(Q))$, nonzero partial isometry $w\in r(\emm\bten \emm) t$ and $\ast$-isomorphism onto its image $\phi: r\De(\El(Q))r\rightarrow \mathcal C:= \phi(r\De(\El(Q))r)\subseteq t(\emm \bten \El(\hat N_1\rtimes diag(Q)))t$ such that \begin{equation}\phi(x)w=wx \text{ for all }  x\in r\De(\El(Q))r.\end{equation}
Since $\El(Q)$ is a factor we can assume without loss of generality that $r=\De(r_1\otimes r_2)$ where $r_1\in \El(Q_i) $ and $r_2\in \El(\hat Q_i)$. Hence $\mathcal C= \phi(r \De(\El(Q))r)= \phi(\Delta (r_1 \El(Q_i)r_2))\bten r_2 \El(\hat Q_i) r_2=: \mathcal C_1\vee \mathcal C_2$ where we denoted by $\mathcal C_1= \phi(\Delta (r_1 \El(Q_1))r_1)\subseteq t(\emm\bten \El(\hat N_1 \rtimes {\rm diag}(Q)))t$ and $\mathcal C_2= \phi(\Delta (r_2 \El(\hat Q_1))r_2)\subseteq t(\emm\bten \El(\hat N_1 \rtimes {\rm diag}(Q)))t$. Notice that $\mathcal C_i$'s are commuting property (T) subfactors of $\emm \bten \El(\hat N_1 \rtimes {\rm diag} (Q))$. Since $N_i \rtimes Q$ is hyperbolic relative to $\{Q\}$ and seeing $\mathcal C_1\vee \mathcal C_2 \subseteq \emm \bten \El(N_i\rtimes {\rm diag} (Q))\subset \El(\G_1\times \G_2\times\cdots\times\G_m \times \hat{\G_1})$ then by applying Theorem \ref{controlprodpropt2} we have that there exits $i\in 1,2$ such that \begin{enumerate}
\item [9)] $\mathcal C_i \prec_{\bar \emm \bten \bar \emm} \El(\G_1\times \G_2\times\cdots\times \G_m\times\hat \G_{1,2})$ or 
\item [10)] $\mathcal C_1\vee \mathcal C_2 \prec_{\bar \emm\bten \bar \emm_1} \El(\G_1\times \G_2\times \times\cdots\times \G_m\times\hat \G_{1,2} \times Q)$.
\end{enumerate}
Since $\mathcal C_i\subset \emm\bten \emm$ then 9) and \cite[Lemma~2.6]{CDK19} imply that $\mathcal C_1\prec_{\emm \bten \emm} \emm\bar \otimes\emm_1 \bar \otimes 1$ which by Lemma \ref{io11gen} further implies that $\mathcal C_i$ is atomic, which is a contradiction. Thus we must have 10). However since $\mathcal C_1\vee \mathcal C_2\subset \emm \bten \emm$ then 10) and \cite[Lemma~2.6]{CDK19} give that $\mathcal C_1\vee \mathcal C_2 \prec_{\emm \bten \emm} \emm\bten \El({\rm diag} (Q))$ and composing this intertwining with $\phi$ (as done in the proof of the first case in Theorem \ref{controlprodpropt2}) we get that  $\De(\El(Q)) \prec_{\emm\bten \emm} \emm\bten \El({\rm diag} (Q))$. Now we show the moreover part. So in particular the above intertwining shows that we can assume from the beginning that $\mathcal C=\mathcal C_1 \vee \mathcal C_2 \subset t(\emm\bten \El({\rm diag} (Q))) t$. Since $Q_i$ are biexact, weakly amenable then by applying \cite[Theorem 1.4]{PV12} we must have that either $\mathcal C_1 \prec \emm\bten \El({\rm diag} (Q_1))$ or $\mathcal C_2 \prec \emm\bten \El({\rm diag} (Q_1))$ or $\mathcal C_1\vee \mathcal C_2$ is amenable relative to  $\emm\bten \El({\rm diag} (Q_1))$ inside $\emm\bten \emm$. However since $\mathcal C_1 \vee \mathcal C_2$ has property (T) the last case above still entails that $\mathcal C_1\vee \mathcal C_2 \prec \emm\bten \El({\rm diag} (Q_1))$ which completes the proof.  \end{proof}

  \begin{theorem}\label{toproductgroupcorners}Let $\G$ be a group as in Notation \ref{semidirectt} and assume that $\Lambda$ is a group such that $\El(\Gamma)=\El(\Lambda)=\emm$. Let $\Delta: \emm\rar \emm\bar \otimes \emm$ be the ``commultiplication along $\Lambda$'' i.e.\ $\Delta (v_\lambda)=v_\lambda \otimes v_\lambda$. Also assume for every $j\in 1,2$ there is $i\in 1,2$ such that either $\Delta (\El(Q_{\bar i}))\prec_{\emm\bar\otimes \emm} \emm\bar\otimes \El( Q_{\bar j})$ or $\Delta (\El(Q_{\bar i}))\prec_{\emm\bar\otimes \emm} \emm\bar\otimes \El(N_{\bar j})$. Then one can find subgroups $\Phi_1,\Phi_2 \leqslant \Phi\leqslant \La $ such that \begin{enumerate}
  \item $\Phi_1,\Phi_2$ are infinite, commuting, property (T), finite-by-icc groups;
  \item $[\Phi:\Phi_1\Phi_2]<\infty$ and $\mathcal{QN}^{(1)}_\La(\Phi)=\Phi$;
  \item there exist $\mu \in \mathcal U(\emm)$, $z\in \mathcal P(\mathcal Z(\El(\Phi)))$, $h= \mu z\mu^*\in \mathcal P (\El(Q))$ such that 
  \begin{equation} \mu \El(\Phi) z\mu^* = h \El(Q)h.
  \end{equation}
  \end{enumerate}
  
  \end{theorem}

\begin{proof} For the proof of this is exactly same proof as \cite[Theorem~5.6]{CDK19}, which is built upon the strategy used in the proof of \cite[Claim 5.2]{CU18}. We encourage the reader to consult this result beforehand as we will focus mainly on the new aspects of the technique. By hypothesis, using \cite[Theorem 4.1]{DHI16} (see also \cite[Theorem 3.1]{Io11} and \cite[Theorem 3.3]{CdSS15}), one can find a subgroup $\Sigma <\La$ with $C_\La(\Sg)$ non-amenable such that $\El(\bar Q_{ 1})\prec_\emm \El(\Sg)$. Also recall that $Q<\G$ is malnormal and has property (T). Note that the group $\G\times\G$ is bi-exact relative to $\mathcal K:=\{ \G\times\bar \G_{ i},\bar \G_{j}\times \G |i,j\in\{1,m \}\} $.
	Let $\Omega=vC_\La(\Sg)$. Let  $\{\mathcal O_1, ..., \mathcal O_k, ... \}$ be a countable enumeration of the finite orbits under conjugation by $\Sg$, and note that $\cup_k \mathcal O_k = \Omega$. Finally, let $\Omega_k=\langle \mathcal O_1, ..., \mathcal O_k \rangle \leqslant \La$, and note that $\Omega_k \nearrow \Omega.$
Then using the same argument from \cite[Claim 5.2]{CU18} one can find nonzero projections $a\in \El(Q_1)$, $q\in \El(\Omega_k)$ a nonzero partial isometry $w\in \El(Q)$ a subalgebra $\mathcal D\subseteq \eta q \El(\Omega_k) qz\mu^*$ and a $\ast$-isomorphism $\phi: a\El(Q_1)a\rar \mathcal D$ such that 
\begin{enumerate}
\item [4)]\label{finiteindex1}$\mathcal D\vee \mathcal D'\cap \eta q \El(\Omega_k) qz \eta^*\subseteq \eta q \El(\Omega_k) qz \eta^*$ is finite index and 
\item [5)]\label{intrel1}$\phi(x)w=wx \text{ for all } x\in a\El(Q_1)a.$ 
\end{enumerate}
Let $r= \eta qz \eta^*$, $ww^*\in \mathcal D'\cap r\El(Q)r$, $w^*w\in (a\El(Q_1)a)'\cap  a\El(Q)a=\El(\bar Q_{ 1})\otimes \mathbb C a$. Thus exist $b\in \El(\bar Q_{ 1})$ projection st $w^*w= a\otimes b$. Pick $c\in U(\El(Q))$such that $w=c(a\otimes b)$ then \eqref{intrel1} gives that 
\begin{equation}\label{equality1}\mathcal Dww^*=w\El(Q_1)w^*=c( aL(Q_1)a\otimes \mathbb C b)c^*.
\end{equation}
Moreover, the same argument from the proof of \cite[Claim 5.2]{CU18}
shows that we can assume the following is a finite inclusion of II$_1$ factors
\begin{equation}\label{finiteindex2}
\mathcal D\subseteq \eta q \El(\Omega_k)qz\eta^*
\end{equation}  
Thus if we denote by $\Xi= \mathcal{QN}_{\Lambda}(\Omega_k)$ using \eqref{equality1} and \eqref{finiteindex2}
 above we see that \begin{equation}\label{equality3}
c(a\otimes b) L(Q) (a\otimes b) c^*= ww^* \eta q z \mathcal {QN}_{\El(\La)}(\El(\Omega_k))''qz\eta^*ww^*= ww^* \eta q z L(\Xi)qz\eta^*ww^*
\end{equation}
and also 

\begin{equation}\label{equality2}
\begin{split}
c(\mathbb Ca \otimes b\El(Q_{\bar 1})b )c^*&= (c(a\El(Q_1)a\otimes \mathbb Cb)c^*)'\cap c(a\otimes b) \El(Q) (a\otimes b) c^*\\
&= (\mathcal Dww^*)'\cap ww^* \eta qz \El(\Xi) qz\eta^* ww^*= ww^*(\mathcal D'\cap \eta qz \El(\Xi) qz\eta^*) ww^*.
\end{split}
\end{equation}
Now we proceed exactly as the proof of \cite[Theorem~5.6]{CDK19} and obtain that there exist subgroups $\Phi_1,\Phi_2\leq \Phi\leq \La$ satisfying following properties: 
 \begin{enumerate}
	\item $\Phi_1,\Phi_2$ are infinite, commuting, property (T), finite-by-icc groups;
	\item $[\Phi:\Phi_1\Phi_2]<\infty$ and $\mathcal{QN}^{(1)}_\La(\Phi)=\Phi$;
	\item there exist $\mu \in \mathcal U(\emm)$, $z\in \mathcal P(\mathcal Z(\El(\Phi)))$, $h= \mu z\mu^*\in \mathcal P (\El(Q))$ such that 
	\begin{equation} \mu \El(\Phi) z\mu^* = h \El(Q)h.
	\end{equation}
\end{enumerate}

\end{proof}


\begin{lemma}\label{somefiniteintersection}Let $\G$ be a group as in Notation \ref{semidirectt} and assume that $\Lambda$ is a group such that $\El(\Gamma)=\El(\Lambda)=M$.  Also assume there exists a subgroup  $\Phi< \La $, a unitary $\mu \in \mathcal U(\emm)$ and projections  $z\in \mathcal Z(\El(\Phi))$, $r= \mu z\mu^*\in \El(Q)$ such that 
	\begin{equation}\label{equalcorner2} \mu \El(\Phi) z\mu^* = r \El(Q)r.
	\end{equation}
	For every $\lambda \in \La\setminus \Phi$ so that $\left |\Phi \cap \Phi^\lambda \right|=\infty$ we have $zu_\lam z=0$. In particular, there is $\lam_o\in \La\setminus \Phi$ so that $\left |\Phi \cap \Phi^{\lambda_o} \right|<\infty$.
	
\end{lemma}

\begin{proof} Since $Q<\G=(N_1\times N_2\times\cdots\times N_m)\rtimes Q$ is almost malnormal, the proof of this lemma is exactly same as \cite[Lemma~5.7]{CDK19}.  \end{proof}

\begin{theorem}\label{parabolicQ} Assume the same conditions as in Theorem \ref{toproductgroupcorners}. Then one can find subgroups $\Phi_1,\Phi_2 \leqslant \Phi\leqslant \La $ so that \begin{enumerate}
		\item $\Phi_1,\Phi_2$ are infinite, icc, property (T) groups so that $\Phi=\Phi_1\times \Phi_2$;
		\item $\mathcal{QN}^{(1)}_\La(\Phi)=\Phi$;
		\item There exists $\mu \in \mathcal U(\emm)$ such that  $\mu \El(\Phi) \mu^* =  \El(Q)$.
	\end{enumerate}

\end{theorem}
\begin{proof} From Theorem \ref{toproductgroupcorners} there exist subgroups $\Phi_1,\Phi_2 \leqslant \Phi\leqslant \La $ such that \begin{enumerate}
		\item $\Phi_1,\Phi_2$ are, infinite, commuting, finite-by-icc, property (T) groups so that $[\Phi:\Phi_1\Phi_2]<\infty$;
		\item $\mathcal{QN}^{(1)}_\La(\Phi)=\Phi$;
		\item There exist $\mu \in \mathcal U(\emm)$ and  $z\in \mathcal P(\mathcal Z(\El(\Phi)))$ with  $h= \mu z\mu^*\in \mathcal P (\El(Q))$ satisfying 
		\begin{equation}\label{maxcorner1} \mu \El(\Phi) z\mu^* = h \El(Q)h.
		\end{equation}
	\end{enumerate}  
	
	Next observe that in \eqref{maxcorner1} we can pick $z\in \mathcal Z(\El(\Phi))$ maximal with the property that for every projection $t \in \mathcal Z(\El(\Phi)z^{\perp})$ we have \begin{equation}\label{nonitertwiningcomplement}L(\Phi_i)t \nprec_\emm \El(Q)\text{ for }i=1,2.\end{equation} 
	
	This follows from the proof of \cite[Theorem~5.8]{CDK19}.

\vskip 0.05in
Next fix $t\in \mathcal Z(\El(\Phi)z^{\perp})$. Since $\El(\Phi_1)t$ and $\El(\Phi_2)t$ are commuting property (T) von Neumann algebras then using the same arguments as in the first part of the proof of Theorem \ref{commutationcontrolincommultiplication} there are two possibilities: either i) there exists $j\in 1,2$ such that  $\El(\Phi_j)t\prec_\emm \El(\bar N_1)$ or ii) $\El (\Phi)t \prec_\emm \El(\bar N_1\rtimes diag(Q))$.  Note that $i)$ implies our desire result by arguing exactly same the proof of \cite[Theorem~5.8]{CDK19}. All we need to show is that case $ii)$ is impossible.  Assuming ii), Theorem \ref{controlprodpropt1} would further imply that either i) $\El(\Phi_j)t\prec_\emm \El(\bar N_{1,2})$ or ii) $\El (\Phi)t \prec_\emm \El(\bar N_{1,2}\rtimes diag(Q))$. Note that again $i)$ implies our desire result by arguing exactly same the proof of \cite[Theorem~5.8]{CDK19} and we are left with the case ii). Continuing this way for $m-1$ times if necessary, we finally end up getting either  i) $\El(\Phi_j)t\prec_\emm \El(\bar N_{1,2,\cdots,m-1})=\El(N_m)$ or ii) $\El (\Phi)t \prec_\emm \El(\widehat N_{\{1,2,\cdots,m-1\}}\rtimes diag(Q))=\El(N_m\rtimes Q)$. Arguing same as above we are done when we have i). Now assuming ii), Theorem \ref{controlprodpropt1} for $n=1$ would imply the existence of $j \in 1,2$ so that $\El(\Phi_j)t \prec_\emm \El(Q)$ which obviously contradicts the choice of $z$.

 \end{proof}

\begin{theorem} In the Theorem \ref{commutationcontrolincommultiplication} we cannot have case 4a).\end{theorem}

\begin{proof} Since $\mu \El(\Phi)\mu^*=\El(Q)$ and $Q_i$ are biexact then by the product rigidity in \cite[Theorem~4.14]{CdSS15} we can assume that there is a unitary $v\in \El(Q)$ and a subset $F\subseteq \{1,2,\cdots,m\}$  such that $v\El(Q_F)v^*=\El(\Phi_1)^t$ and $v\El(Q_{F^C})v^*=\El(\Phi_2)^{1/t}$. Assume by contradiction that for all $j\in\{1,\cdots,m\}$ and any subset $S\subseteq \{1,\dots,m\}$ we have $\Delta (\El(Q_S))$  $\prec_{\emm\bar\otimes \emm} \emm\bar\otimes \El(\widehat N_j)$ or $\Delta(\El( Q_{S^C}))$ $\prec_{\emm\bar\otimes \emm} \emm\bar\otimes \El(\widehat N_j)$. Without loss of generality we assume that $\Delta (\El(Q_F))$  $\prec_{\emm\bar\otimes \emm} \emm\bar\otimes \El(\widehat N_j)$. Using \cite[Theorem 4.1]{DHI16} and the property (T) on $\widehat N_j$ one can find a subgroup $\Sg<\La$ such that $\El(Q_F)\prec_\emm \El(\Sg)$ and $\El(\widehat N_j)\prec_\emm \El(C_\La(\Sg))$.  Thus we get that $\El(\Phi_i)\prec
_\emm \El(\Sg)$  and hence $[\Phi_i: g\Sg g^{-1} \cap \Phi_i]<\infty$. So working with $g\Sg g^{-1}$ instead of $\Sg$ we can assume that  $[\Phi_i: \Sg  \cap \Phi_i]<\infty$.In particular $\Sg\cap \Phi_i$ is infinite and since $\Phi$ is almost malnormal in $\La$ it follows that $C_\La(\Sg\cap \Phi_i) < \Phi$. Thus we have that  $\El(\widehat N_j)\prec_\emm \El(C_\La(\Sg))\subseteq \El(C_\La(\Sg\cap \Phi_i))\subset \El(\Phi)=\mu^*\El(Q)\mu$ which is a contradiction. \end{proof}

\begin{theorem}\label{commutationcontrolincommultiplication2} Let $\G$ be a group as in Notation \ref{semidirectt} and assume that $\Lambda$ is a group such that $\El(\Gamma)=\El(\Lambda)=\emm$. Let $\Delta: \emm \rar \emm \bar \otimes \emm$ be the comultiplication ``along $\Lambda$'' i.e. $\Delta (v_\lambda)=v_\lambda \otimes v_\lambda$. Then for any proper subset $S\subseteq \{1,\cdots,m\}$ the following hold:
	\begin{enumerate}
		\item [i)]  $\Delta(\El(N_S\times \widehat N_S)\prec^s_{\emm \bar\otimes \emm} \El(N)\bar\otimes \El(N)$, and
		\item [ii)] there is a unitary $u\in \emm \bten \emm$ such that $u\Delta(\El(Q))u^*\subseteq  \El(Q)\bar\otimes \El(Q)$.    
	\end{enumerate}
	
\end{theorem}

\begin{proof} First we show i). From Theorem \ref{commutationcontrolincommultiplication} we have that for all $j\in \{1,\cdots,m\}$ and for any proper subset $F\subseteq \{ 1,\cdots,m\}$ we have either $\Delta(\El(N_{F}))\prec_{\emm \bar\otimes \emm} \emm\bar\otimes \El(\widehat N_j)$ or $\Delta(\El(\widehat N_{F}))\prec_{\emm \bar\otimes \emm} \emm\bar\otimes \El(\widehat N_j)$. Notice that since for every subset $F\subseteq \{1,\cdots,m\}$ we have $\mathcal N_{\emm \bten \emm}\De(\El(N_F))''\supset \De(\emm)$ and $\De(\emm)'\cap \emm \bten \emm =\mathbb C 1$ then by \cite[Lemma 2.4 part (3)]{DHI16} we actually have either  $\Delta(\El(N_{F}))\prec^s_{\emm\bar\otimes \emm} \emm\bar\otimes \El(\widehat N_j)$ or $\Delta(\El(\widehat N_{F}))\prec^s_{\emm\bar\otimes \emm} \emm\bar\otimes \El(\widehat N_j)$. If we have the following situation:
\begin{align*}
\Delta(\El(N_{F}))\prec^s_{\emm\bar\otimes \emm} \emm\bar\otimes \El(\widehat N_i)\\
and\ \Delta(\El(\widehat N_{F}))\prec^s_{\emm\bar\otimes \emm} \emm\bar\otimes \El(\widehat N_j)
\end{align*}
, for some i and j, then we are done. 	
Notice that if for all $j\in\overline{1,k}$ we have $\Delta(\El(N_{F}))\prec^s_{\emm\bar\otimes \emm} \emm\bar\otimes \El(\widehat N_j)$. Then we would have $\De (\El(N_{F})\prec^s \emm \bar\otimes \El(\widehat N_j)$ for all $j$ and hence by \cite[Lemma 2.8 (2)]{DHI16} we get that $\De (\El(N_{F})\prec^s \emm \bar\otimes \El(\cap_j\widehat N_j)=\emm \bar\otimes 1$ which is a contradiction. Furthermore using the same arguments as in \cite[Lemma 2.6]{Is16} we have that $\Delta(\El(N_F\times \widehat N_F)\prec^s_{\emm \bar\otimes \emm} \emm \bar\otimes \El(N)$. Then working on the left side of the tensor we get that $\Delta(\El(N_F\times \widehat N_F)\prec^s_{\emm \bar\otimes \emm} \El(N)\bar\otimes \El(N)$.

	\vskip 0.05in
	
Part ii) follows from the proof of part ii) of \cite[Theorem~3.10]{CDK19}.

\end{proof}

\subsubsection{Proof of Theorem \ref{semidirectprodreconstruction}} 

\begin{proof} The proof has two parts. Firstly we are going to reconstruct the acting group $Q$. This follows from the proof of the first part of \cite[Theorem~5.1]{CDK19}.
	From Theorem \ref{commutationcontrolincommultiplication2} have that $\Delta(\El(N_1 \times N_2\times\cdots \times N_m)) \prec^s_{\emm \tp \emm} \El(N) \tp \El(N)$. Proceeding exactly as in the proof of \cite[Claim 4.5]{CU18} we can show that $\Delta(\mathcal A) \subseteq \mathcal A \tp \mathcal A$, where $\mathcal A= u\El(N)u^{\ast}$. By \cite[Lemma~2.8]{CDK19}, there exists a subgroup $\Sg < \La$ such that $\mathcal A= \El(\Sg)$. The last part of the proof of \cite[Theorem 5.2]{CU18} shows that $\La= \Sg \rtimes \Phi$. In order to reconstruct the product feature of $\Sigma$, we need a couple more results.
	
	\begin{claim} For every $j\in \overline{1,m}$ there exists  $i\in \overline{1,m}$ such that \begin{equation}\label{intonetensor}
		\Delta(\El(\widehat N_j))\prec^s \El (N) \bten \El
		(\widehat N_i).
		\end{equation}
	\end{claim} 
	
	\noindent \textit{Proof of Claim.} Note that for every $i,j\in\overline{1,m}$ we have either $\Delta(\El(N_{j}))\prec^s_{\emm\bar\otimes \emm} \emm\bar\otimes \El(\widehat N_i)$ or $\Delta(\El(\widehat N_{j}))\prec^s_{\emm\bar\otimes \emm} \emm\bar\otimes \El(\widehat N_i)$. If for some $j$, we have  $\Delta(\El(N_{j}))\prec^s_{\emm\bar\otimes \emm} \emm\bar\otimes \El(\widehat N_i)$ for every $i\in\overline{1,m}$. Then by \cite[Lemma~2.8 (2)]{DHI16} we get that $\Delta(\El(N_{j}))\prec^s_{\emm\bar\otimes \emm} \emm\bar\otimes \El(\cap_i \widehat N_i)=\emm\bten 1$ which is a contradiction by \cite[Proposition~7.2.1]{IPV10}. Hence for every  $j\in \overline{1,m}$ there exists  $i\in \overline{1,m}$ such that $\Delta(\El(\widehat N_j))\prec^s \emm \bten \El
	(\widehat N_i)$. Since $\Delta(\El(\widehat N_j))\prec^s_{\emm\bten\emm}\El(N)\bten \El(N)$, we get the claim by \cite[Lemma~2.8 (2)]{DHI16}.

	$\hfill\blacksquare$
	\vskip 0.07 in
	Let $\mathcal A_j= u\El(\widehat N_j))u^{\ast}$. Thus, we get that $\Delta(\mathcal A) \prec^s \El(N) \otimes \El(\widehat N_i)$ for some $i\in\overline{1,m}$. This implies that for every $\varepsilon >0$, there exists a finite set $S \subset u^{\ast} Q u$, containing $e$, such that $\|d-P_{S \times S}(d)\|_2 \leq \varepsilon$ for all $d \in \Delta (\mathcal A_j)$. However, $\Delta (\mathcal A_j)$ is invariant under the action of $u^{\ast}Q u$, and hence arguing exactly as in \cite[Claim 4.5]{CU18} we get that $\Delta(\mathcal A_j) \subset (\El(\Sg) \bten u\El(\widehat N_i)u^{\ast})$. We now separate the argument into two different cases:
	
	\textbf{Case I:} $i=j$.
	
	In this case, $\Delta(\mathcal A_j) \subseteq \El(\Sg) \tp \mathcal A_j$. Thus by \cite[Lemma~2.8]{CDK19} we get that there exists a subgroup $\Sg_0 < \Sg$ with $\mathcal A_j = \El(\Sg_j)$. Now, $\mathcal A_j' \cap \El(\Sg)= u\El(N_i)u^{\ast}$. Thus, $\El(\Sg_j)' \cap \El(\Sg) = u\El(N_i)u^{\ast}$. Note that $\Sg$ and $\Sg_j$ are both icc property (T) groups. This implies that $\El(\Sg_j)' \cap \El(\Sg)= \El(vC_{\Sg}(\Sg_j))$, where $vC_{\Sg}(\Sg_0)$ denotes the \textit{virtual centralizer} of $\Sg_j$ in $\Sg$. Proceeding as in \cite{CdSS17} we can show that $\Sg= \Sg_j \times \Sg_j^1$.
	
	\textbf{Case II:} $i\neq j$.
	
	Let $\mathcal B= u\El(\widehat N_i)u^{\ast}$. In this case, $\Delta(\mathcal A) \subseteq \El(\Sg) \tp \mathcal B$. However, \cite[Lemma~2.8]{CDK19} then implies that $\mathcal A \subseteq \mathcal B$, which is absurd, as $\El(\widehat N_i)$ and $\El(\widehat N_j)$ have orthogonal subalgebras namely $\El(N_j)$ and $\El(N_i)$ respectively. Hence this case is impossible and we are done. \end{proof}



\section{Semidirect product rigidity for class $\mathscr{S}_{\infty,n}$}

In this section we study algebraic properties of groups arising from class $\mathscr S_{\infty,m}$. As explained in subsection~2.4, groups in this class appear as inductive limits of groups from class $\mathscr S_{m,n}$. We showed that semidirect product feature for this class of groups survives the passage to group von Neumann algebra regime. Before proving the theorem, we need following auxiliary theorems about location of acting group and core under comultiplication. 

\begin{theorem}\label{semidirectrigiditys-infty1}
	Let $\G\in\mathscr{S}_{\infty,n}$ such that $\El(\G)=\El(\La)=\emm$. Let $\Delta:\emm\rightarrow \emm\bten\emm$. Then the following hold: 
	\begin{enumerate}
		\item[a.] $\Delta(\El(N))\prec^s_{\emm\bten\emm} \El(N)\bten\El(N)$. 
		\item[b.] There is a unitary $u\in\emm\bten\emm$ such that $u(\Delta(\El(Q)))u^{\ast}\subseteq \El(Q)\bten\El(Q)$.
	\end{enumerate}
\end{theorem}

\begin{proof}
	Let $\G=(N_1\times N_2\times \cdots )\rtimes (Q_1\times\cdots\times Q_n)$. For every finite set $F$ denote $N_F=\underset{i\in F}{\times} N_i$ and $G_F=N_F\rtimes Q$. When $F$ is a singleton set $\{i\}$ we simply denote   $G_F$ by $G_i$. Note that $G_F$ has property (T) for all finite set $F$ and $\G=\bigcup G_F$. Also note that $\emm\bten\emm$ is union of property (T) factors of the form $\El(G_F)\bten\El(G_{F'})$ for $F,F'$ finite. Hence for every finite set $F$ there is a finite set $F'$ such that we have;
	\begin{align}
		\Delta(\El(G_F))\prec_{ \emm\bten \emm}^s \El(G_{F_1})\bten \El(G_{F_2}) \ for \ some\ finite\ sets\ F_1\ and \ F_2.
	\end{align} 
Since $\El(G_{F_1})\bten \El(G_{F'})\subseteq \emm\bten \El(G_{F'})$ we get that $\Delta(\El(G_F))\prec_{ \emm\bten\emm}\emm\bten \El(G_{F'})$. Since $\El(G_F)$ is irreducible in $\emm$ we get that $\Delta(\El(G_F))\prec_{ \emm\bten\emm}\emm\bten \El(G_{F'})$. Let $\emm\bten\emm=\tilde{\emm}$. Using the fact that $\Delta(\El(N_F))\subseteq \Delta(\El(G_F))$ and the normality of $N_F\in\G$ we conclude that $\Delta(\El(N_F))\prec_{\tilm}\El(G_{F_1})\bten \El(G_{F2})$. Hence there exists a nonzero projection $r\in \Delta(\El(N_F)),t\in \El(G_{F_1})\bten \El(G_{F_2})$, a nonzero partial isometry $w\in r(\tilm)t$ and a $\ast$-isomorphism $\phi:r(\Delta(\El(N_F)))r\rightarrow  \Bee:=\phi(r(\Delta(\El(N_F)))r)\subseteq t(\El(G_{F_1})\bten\El(G_{F_2}))t$ such that 
\begin{align}
	\phi(x)w=wx\ for\ all\ x\in r(\Delta(\El(N_F)))r
\end{align}  
Since $\El(N_F)$ is a factor we can assume that $r=\Delta(r_1\otimes r_2)$ where $r_1\in \El(N_i)$ and $r_2\in\El(\widehat N_i)$. Note that here $\widehat N_i=\underset{j\neq i,j\in F}{\times} N_j$. Hence $\Bee=\phi(r_1(\El(N_i))r_1\bten r_2(\El(\widehat N_i))r_2  )=: \Bee_1\vee\Bee_2$. Where we denote $\Bee_1=r_1(\El(N_i))r_1$ and $\Bee_2=r_2(\El(\widehat N_i))r_2$. Notice that $\Bee_1$ and $\Bee_2$ are two commuting property (T) subalgebra of $\El(G_{F_1})\bten\El(G_{F_2})$.  Without loss of generality let us assume that $F=\{1,\cdots,m \},F_1=\{1,\cdots,l\}$ and $F_2=\{1,\cdots,k\}$. Observe that $G_{F_1}\times G_{F_2}<(G_1\times\cdots\times G_l)\times (G_1\times \cdots\times G_k)=:H\times K$. Denote $\El(H\times K)=\enn$. Hence we have $\Bee_1$ and $\Bee_2$ are two commuting property (T) subalgebra of $\enn $ where $G_i$'s are hyperbolic relative to $Q$. By applying \cite[Theorem~5.3]{CDK19} we get that for every $j\in\overline{1,k}$ one of the following must hold:
\begin{enumerate}
	\item[1)] There exists $i\in\{1,2\}$ such that $\Bee_i\prec_{ \enn}\El(H\times\widehat K_j) $;
	\item[2)] $\Bee=\Bee_1\vee\Bee_2\prec_{ \enn}\El(H\times \widehat K_j\times Q)$,
\end{enumerate}   
Where $\widehat K_j=\underset{s\neq j,s\in \overline{1,k}}{\times } G_s$. 

Now following the proof of Theorem~\ref{commutationcontrolincommultiplication} we get that for every $j\in\overline{1,k}$ there is $i\in\{1,2\}$ $\Bee_i\prec_{\enn} \El(H)\bten \El(\widehat N_j)$. Now composing the intertwining we get that for every $i\in\overline{1,l},j\in\overline{1,k}$ we have $\Delta(\El(N_{i}))$ or $\Delta(\El(\widehat N_i))\prec_{\El(H)\bten\El(K) }\El(H)\bten \El(\widehat N_j)$. If for some $j$ we have $\Delta(\El(N_{i}))\prec_{\El(H)\bten\El(K) }\El(H)\bten \El(\widehat N_j)$ for all $i$, then we get $\Delta(\El(N_{i}))\prec^s_{\El(H)\bten\El(K) }\El(H)\bten \El(\widehat N_j)$ for all $i$ since $\Delta(N_i)$ is normal. Hence by \cite[Corollary~1]{CDHK20} we get that $\Delta(\El(N_F))\prec^s_{\El(H)\bten\El(K) }\El(H)\bten \El(\widehat N_j)\subseteq \emm\bten \El(N)$. Hence we have $\Delta(\El(N_F))\prec^s_{\emm\bten\emm} \emm\bten \El(N)$. If above is not the case then for every $j$ we have some $i_j$ such that $\Delta(\El(\widehat N_{i_j}))\prec_{\El(H)\bten\El(K) }\El(H)\bten \El(\widehat N_j)$. Then by taking $j_1\neq j_2$ we can conclude that $\Delta(\El(N_{i}))(\subseteq \Delta(\widehat N_{i{j_1}})\cup \Delta (\widehat N_{i{j_2}}))\prec_{\El(H)\bten\El(K) }\El(H)\bten \El(\widehat N_j)$ for all $i$. Since if for every $j_1\neq j_2$ we have $i_{j_1}=i_{j_2}$, then we get that $\Delta (\widehat N_{i{j_1}})\prec^s_{\El(H)\bten\El(K)} \El(H)\bten \El(\widehat N_j)$ for every $j$. Hence by \cite[Lemma~2.8(2)]{DHI16} we get that $\Delta (\widehat N_{i{j_1}})\prec^s_{\El(H)\bten\El(K)}\El(H)\bten 1$. This is a contradiction to Lemma~\ref{io11gen}. Hence we get that for every finite set $F$, $\Delta(\El(N_F))\prec^s_{\emm\bten\emm} \emm\bten\El(N)$. This implies that $\Delta(\El(N_F))\lessdot \emm\bten\El(N)$ for every finite set $F$. By \cite[Lemma~2.7]{DHI16} Lemma~2.6 we get that $\Delta(\El(N))\lessdot \emm\bten\El(N)$. Hence by \cite[Theorem~1.4]{PV12} and normality of $N$ we get that $\Delta(\El(N))\prec^s_{ \emm}\emm\bten\El(N)$. Similarly working with the left side of the tensor we get $\Delta(\El(N))\prec^s_{\emm\bten\emm} \El(N)\bten \emm$. Now by \cite[Lemma~2.8(2)]{DHI16} we get that $\Delta(\El(N))\prec^s_{\emm\bten\emm} (\emm\bten\El(N))\cap (\El(N)\bten \emm)=\El(N)\bten\El(N)$. This proves part a. of the Theorem. 

In order to show part b. we first show that there is an unitary $u\in \emm\bten\emm$ such that $u(\Delta(\El(Q)))u^{\ast}\subseteq \emm\bten\El(Q)$. First note that by doing similar argument as above together with second part of Theorem~\ref{commutationcontrolincommultiplication}, we can show that $\Delta(\El(Q))\prec_{ \emm \bten \emm} \enn\bten\El(Q)$. That is, there are projections $p\in\Delta(\El(Q)),q\in\emm\bten\El(Q)$, a nonzero partial isometry $v\in q(\emm\bten\El(Q))q$ and a $\ast$-isomorphism $\phi:p(\Delta(\El(Q)))p\rightarrow \mathcal C:=\phi(p(\Delta(\El(Q)))p)\subseteq q(\emm\bten\El(Q))q$ such that;
\begin{align}\label{actingdiscritise}
	\phi(x)v=vx\ for\ all\ x\in p\Delta(\El(Q))p.
\end{align} 
We also have $vv^*\in\mathcal C'\cap q(\emm\bten \emm)q$ and $v^*v\in p\Delta(\El(Q))p'\cap p(\emm\bten\emm)p$ and moreover we can assume that the support $s(E_{\emm\bten\El(Q)}(vv^{\ast}))=q$. Note that $\mathcal C\nprec_{ \emm\bten\El(Q)}\emm\bten 1$, as if $\mathcal C\prec_{ \emm\bten\El(Q)}\emm\bten 1$, by the first part of the \cite[Theorem~5.3]{CDK19} this would imply that $\El(Q)\prec_{ \emm\bten\El(Q)}\emm\bten 1$ which contradicts \cite[Proposition~7.2.2]{IPV10}. Since $Q<\G$ is malnormal, by \cite[Lemma~2.2]{CDK19} we get that $vv^*\in \mathcal C'\cap q(\emm\bten \emm)q\subseteq \mathcal C'\cap p(\emm\bten\El(Q))p$ and hence by \ref{actingdiscritise} we get that $vp\Delta(\El(Q))pv^*=\mathcal Cvv^*\subseteq \emm\bten \El(Q)$. Now we get a unitary $u\in \emm\bten\emm$ such that $u\Delta(\El(Q))u^*\subseteq \emm\bten\El(Q)$ as $\emm\bten\El(Q)$ is a factor. Arguing similarly on the other side of the tensor we get that there is a unitary $r\in \emm\bten\emm$ such that $u\Delta(\El(Q))u^*\subseteq \El(Q)\bten\emm$. Hence by applying \cite[Lemma~2.8(2)]{DHI16} again we conclude that $\Delta(\El(Q))\prec^s_{\emm\bten\emm} \El(Q)\bten\El(Q)$. Note that $\Delta(\El(Q))\nprec_{\emm\bten\emm}\emm\bten 1,1\bten\emm$, and hence applying the same argument as above we conclude the desire result to get part b. of the theorem.
\end{proof}

\begin{theorem}\label{acting-productfeature}
	Let $\G\in\mathscr{S}_{\infty,n}$ and $\La$ be any group such that $\emm=\El(\G)=\El(\La)$. With the same comultiplication notation from previous theorem, there exist subgroups $\Phi_1,\Phi_2\leq \Phi\leq \La$ such that 
	\begin{enumerate}
		\item $\Phi_1,\Phi_2$ are icc, commuting, property (T) groups such that $\Phi=\Phi_1\times\Phi_2$;
		\item $\mathcal Q\mathcal N^{(1)}_{\La}(\Phi)=\Phi$; 
		\item There exists $\mu\in\mathcal U(\emm)$ such that $\mu\El(Q)\mu^*=\El(Q)$.
		  
	\end{enumerate} 
\end{theorem}

\begin{proof}
	By Theorem~\ref{toproductgroupcorners},Lemma~\ref{somefiniteintersection} and first part of Theorem~\ref{parabolicQ} we get subgroups $\Phi_1,\Phi_2\leq \Phi\leq \La$ such that $\Phi_1,\Phi_2$ are icc, property (T), commuting, finite-by-icc subgroups of $\La$ such that $[\Phi;\Phi_1\Phi_2]<\infty$ and there exist $\mu \in \mathcal U(\emm)$ and  $z\in \mathcal P(\mathcal Z(\El(\Phi)))$ with  $h= \mu z\mu^*\in \mathcal P (\El(Q))$ satisfying 
	\begin{equation}\label{5} \mu \El(\Phi) z\mu^* = h \El(Q)h.
	\end{equation}

Following the argument in Theorem~\ref{parabolicQ}, in \eqref{5} we can pick $z\in \mathcal Z(\El(\Phi))$ maximal with the property that for every projection $t \in \mathcal Z(\El(\Phi)z^{\perp})$ we have \begin{equation}
	L(\Phi_i)t \nprec_\emm \El(Q)\text{ for }i=1,2.
\end{equation}

\vskip 0.05in
Next fix $t\in \mathcal Z(\El(\Phi)z^{\perp})$. Since $\El(\Phi_1)t$ and $\El(\Phi_2)t$ are commuting property (T) von Neumann algebras of $\emm$, which is an increasing union of property (T) algebras namely $\El(G_F)$ (same notation as in Theorem~\ref{semidirectrigiditys-infty1}). Hence  $\El(\Phi_1)t$ and $\El(\Phi_2)t$ intertwines into some $\El(G_F)$. Now using Theorem~\ref{parabolicQ} in conjunction with \cite[Theorem~5.3]{CDK19} we conclude the theorem.
\end{proof}

\begin{theorem}\label{semidirect-S_inftystructure}
	Let $\G=N\rtimes Q\in\mathscr{S}$, $\La$ be any group and $\Theta:\El(\G)\rightarrow \El(\La)$ be a $\ast$-isomorphism. Then the group $\La$ admits a semidirect product decomposition  $\La=\Sg\rtimes \Phi $ satisfying the following properties: There is a group isomorphism $\delta: Q\rightarrow \Phi$, a character $\zeta:Q\rightarrow \mathbb T$, a $\ast$-homomorphism $\Theta_0:\El(N)\rightarrow \El(\Sg)$ and a unitary $w\in \El(\La)$ such that for every $x\in\El(N)$ and $\g\in Q$ we have 
	$$\Theta(xu_{\g})=\zeta(\g)w\Theta_0(x)v_{\delta(\g)}w^{\ast} .$$
	Where $\{ u_{\g}|\g\in Q\}$ and $\{v_{\lambda}|\lambda\in \Phi\}$ are the canonical unitaries of of $\El(Q)$ and $\El(\Phi)$ respectively.
\end{theorem}

\begin{proof} 
From Theorem~\ref{semidirectrigiditys-infty1} part b. we get that there is a unitary $u\in\emm\bten\emm$ such that $u\Delta(\El(Q))u^{\ast}\subseteq \El(Q)\bten\El(Q)$. Notice that proceeding exactly as in the proof of the first part of \cite[Theorem~5.2]{CDK19}, we get that the height $h_{Q\times Q}(u\Delta(\El(Q))u^{\ast})>0$. This together with Theorem~\ref{acting-productfeature} and \cite[Lemma~2.4,2.5]{CU18} implies that $h_Q(\mu\Phi\mu^{\ast})>0$ and then by \cite[Theorem~3.1]{IPV10} there is a unitary $\mu_0\in\emm$ such that $\mathbb T\mu_0\Phi\mu_0^{\ast}=\mathbb TQ$. 

 By using Theorem~\ref{semidirectrigiditys-infty1} part a. we get that $\Delta(N)\prec_{ \emm \bten \emm}^s\El(N)\bten\El(N)$. Proceeding exactly as in the proof of \cite[Claim~4.5]{CU18} we can show that $\Delta(\Aa)\subseteq \Aa\bten\Aa$ where $\Aa=u\El(N)u^{\ast}$. Then by \cite[Lemma~2.8]{CDK19} there is a subgroup $\Sg\leq \La$ such that $\Aa=\El(\Sg)$. Since $u_{\g}$'s normalizes $\Aa$, it follows that $v_{\delta(\g)}$ normalizes $\Sg$ for all $\g$. Moreover, since $\El(\La)=A\rtimes Q$, then $\La$ admits a semidirect product decomposition $\La=\Sg\rtimes\Phi$. This proves the Theorem.   

\end{proof}

\section{Trivial Fundamental Group arising from class $\mathscr{S}_{m,n}$}

In this section we show that the group von Neumann algebra associated to groups from class $\mathscr S_{m,n}$ has trivial fundamental group for all $m,n\geq 2$. We use similar method from \cite{CDHK20} to locate two commuting property (T) subalgebras and identify acting group upto unitary conjugation. Then we use height computation to conclude the conclusion using \cite{KV15}.   

\begin{theorem}\label{controlcommutingpropTsubalgebratrivialf_group}
	Let $N\rtimes Q\in \mathscr{S}_{m,n}$. Also let $\mathcal A_1,\mathcal A_2\subseteq \El(N\rtimes Q)=\emm$ be two commuting, property (T), type $\rm II_1$ factors. Then for all $k\in\overline{1,m}$ one of the following holds:
	\begin{enumerate}
		\item[i.] There exists $i\in\overline{1,2}$ such that $\mathcal A_i\prec_{ \emm}\El(\widehat N_k)$;
		\item[ii.] $\mathcal A_1\vee\mathcal A_2\prec_{ \emm} \El(Q).$ 
	\end{enumerate} 
\end{theorem}
\begin{proof} 
	\begin{claim}
		We claim that one of the following holds: 
		\begin{enumerate}
			\item[1.] There exists $i\in\overline{1,2}$ such that $\mathcal A_i\prec_{ \emm}\El(\widehat N_k)$;
			\item[2.] $\mathcal A_1\vee\mathcal A_2\prec_{ \emm} \El(\widehat N_k)\rtimes Q.$ 
		\end{enumerate} 
	\end{claim}
	\noindent \textit{Proof of Claim.} Let $\widehat G_k=\widehat{N}_k\rtimes Q$ and $G_k=N_k\rtimes Q$ for all $k\in\overline{1,m}$. Notice that $N\rtimes Q\leq G_k\times \widehat G_k$ where $Q$ embedded as $diag(Q)\leq Q\times Q$. Notice that $\mathcal A_1,\mathcal A_2\subseteq \El(G_k\times \widehat G_k)=:\tilde \emm$. Then by \cite[Theorem~5.3]{CDK19} there exists $i\in\overline{1,2}$ such that
	\begin{enumerate}
		\item[a)] $\mathcal A_i\prec_{\tilde \emm} \El(\widehat G_k)$, or
		\item[b)] $\mathcal A_1\vee\mathcal A_2\prec_{\tilde \emm}\El(\widehat G_k\times Q)$. 
	\end{enumerate}
	Assume $a)$. Since $\mathcal A_1\vee\mathcal A_2\subseteq \El(N)\rtimes Q$, using \cite[Lemma~2.3]{CDK19} we further get that $\mathcal A_i\prec_{\tilde \emm}\El(G\cap h\widehat G_k h^{-1})=\El((N\rtimes Q)\cap (\widehat N_k\rtimes Q))=\El(\widehat N_k)$. Thus we have that $c)$ $\mathcal A_i\prec_{\tilde \emm}\El(\widehat N_k)$.
	
	Assume $b)$. Then $\mathcal A_1\vee\mathcal A_2\prec_{\tilde \emm}\El(G\cap h(\widehat G_k\times Q)h^{-1})=\El(h(\widehat N_k\rtimes diag(Q))h^{-1}).$ This then implies that $d)$ $\mathcal A_1\vee\mathcal A_2\prec_{\tilde \emm} \El(\widehat N_k)\rtimes Q$. Observe that \cite[Lemma~2.5]{CDK19} together with this case $d)$ implies case $2.$ in the statement of the theorem. 
	
	Proceeding exactly as in \cite[Theorem~4.1]{CDHK20} we can see that case $c)$ implies $1.$ in the theorem. This proves the claim. $\hfill\blacksquare$
	
	In order to prove the theorem, we need to show is that $\mathcal A_1\vee\mathcal A_2\prec_{ \emm}\El(\widehat N_k)\rtimes Q$ implies $\mathcal A_1\vee\mathcal A_2\prec_{ \emm}\El(Q)$. Since $\Aa_1 \vee \Aa_2 \prec_{\emm} \El(\widehat N_k)\rtimes Q $, there exists \begin{equation} \label{eq1}
	\psi: p(\Aa_1 \vee \Aa_2)p \rightarrow \psi( p(\Aa_1 \vee \Aa_2)p ) = \mr \subseteq q (\El(\widehat N_k) \rtimes Q)q
	\end{equation}
	$\ast$-homomorphism, nonzero partial isometry $v \in q \emm p$ such that \begin{align}\label{bal2}\psi(x)v =vx \text{ for all }x \in p(\Aa_1 \vee \Aa_2)p.\end{align}
	Notice that we can pick $v$ such that the support projection satisfies  $s(E_{\El(\widehat N_k\rtimes Q)}(vv^*))=q$. Moreover, since $\Aa_i$'s are factors we can assume that $p=p_1p_2$ for some $p_i\in \mathscr P(\Aa_i)$. 
	
	Next let $\mr_i= \psi(p_i \Aa_i p_i)$. Note that $\mr_1$, $\mr_2$ are commuting property (T) subfactors such that $\mr_1 \vee \mr_2= \mr \subseteq q (\El(\widehat N_k) \rtimes Q)q$. Proceeding similarly as in the proof of the last part of Theorem \ref{commutationcontrolincommultiplication} and \cite[Theorem~4.2]{CDHK20} we can conclude that $\mathcal A_1\vee\mathcal A_2\prec_{ \emm}\El(Q)$.  
	
\end{proof}

The following theorem is immediate from Theorem \ref{controlcommutingpropTsubalgebratrivialf_group} and \cite[Theorem~4.3]{CDHK20}. We include the details for readers' convenience.

\begin{theorem} \label{intertwiningthm}
	Let $\Aa_1,\Aa_2\subseteq \El(N)\rtimes Q=\emm$ be two commuting, property (T), type $\rm II_1$ factors such that $(\Aa_1\vee\Aa_2)'\cap r(\El(N)\rtimes Q)r=\mathbb{C}r$. Then one of the following holds:
	\begin{enumerate}
		\item[a)] $\Aa_1\vee\Aa_2\prec^s_{\emm}\El(N)$, or
		\item[b)] $\Aa_1\vee\Aa_2\prec^s_{\emm}\El(Q)$.
	\end{enumerate}
\end{theorem}
\begin{proof}
	Fix $k\in\{1,2\}$. By Theorem~\ref{controlcommutingpropTsubalgebratrivialf_group} we get that either 
	\begin{enumerate}
		\item[i)] $i_k\in\{1,2\}$ such that $\Aa_{i_k}\prec_{\emm} \El(\widehat N_k)$, or
		\item[ii)] $\Aa_1\vee\Aa_2\prec_{\emm} \El(Q)$.
	\end{enumerate}
	Note that case ii) together with the assumption $(\Aa_1\vee\Aa_2)'\cap r(\El(N)\rtimes Q)r=\mathbb{C}r$ and \cite[Lemma 2.4]{DHI16} already give $b)$. So assume that case i) holds. Hence for all $k\in\{1,2\}$, there exists $i_k\in\{1,2\}$ such that $\Aa_{i_k}\prec_{\emm}\El(\widehat N_k)$. Using \cite[Lemma~2.4]{DHI16}, there exists $0\neq z\in\mathcal Z(\enn_{r\emm r}(\Aa_{i_k})'\cap r\emm r)$ such that $\Aa_{i_k}z \prec^s_{\emm}\El(\widehat N_k)$. Since $\Aa_1\vee\Aa_2\subseteq \enn_{r\emm r}(\Aa_{i_k})''$, then $\enn_{r\emm r}(\Aa_{i_k})'\cap r\emm r\subseteq (\Aa_1\vee\Aa_2)'\cap r\emm r=\mathbb{C}r$. Thus we get that $z=r$. In particular 
	\begin{align}\label{1}
	\Aa_{i_k}\prec^s_{\emm} \El(\widehat N_k).  
	\end{align}
	We now briefly argue that $k\neq l\Rightarrow i_k\neq i_l$. Assume by contradiction that $i_1=i_2=i$. Then for $k=1$ we have that $\Aa_i\prec^s_{\emm}\El(\widehat N_1)$ and $\Aa_i\prec^s_{\emm}\El(\widehat N_2)$. By \cite[Lemma 2.6]{DHI16}, this implies that $\Aa_i\lessdot_{\emm} \El(\widehat N_1)$ and $\Aa_i\lessdot_{\emm} \El(\widehat N_2)$. Note that $\El(N_i)$ are regular in $\emm$ and hence by \cite[Proposition 2.7]{PV11} we get that $\Aa_i\lessdot_{\emm} \El(\widehat N_1)\cap\El(\widehat N_2)=\El(\widehat N_{\{1,2\}})$, which implies that $\Aa_i$ is amenable. This contradicts our assumption that $\Aa_i$ has property $(T)$. Thus $i_k\neq i_l$ whenever $k\neq l$. Therefore we have that $\Aa_{i_1}\prec_{\emm}^s\El(N_1)\subseteq \El(N)$ and $\Aa_{i_2}\prec_{\emm}^s\El(N_2)\subseteq \El(N)$. Using \cite[Corollary~1]{CDHK20} we get that $\Aa_1\vee\Aa_2\prec^s_{\emm}\El(N)$, which completes the proof.  
\end{proof}

Next we show that we can identify ``core'' algebras and the algebras associated to the acting groups. The proof heavily relies on Theorem \ref{controlcommutingpropTsubalgebratrivialf_group} and \cite[Theorem~4.4,4.5]{CDHK20}.

\begin{theorem} \label{core-actinggp} 
	Let $N\rtimes Q, M\rtimes P \in \mathscr S_{m,n}$. Let $p\in \El(M\rtimes P)$ be a projection and assume that  $\Theta: \El(N \rtimes Q) \rightarrow p\El(M \rtimes P)p$ is a $\ast$-isomorphism.  Then the following hold
	
	\begin{enumerate}
		\item There exists  $v \in \mathscr U (p\El(M \rtimes P)p)$ such that $\Theta(\El(N))= vp\El(M)pv^{\ast}$, and 
		\item There exists $u \in \mathscr U (\El(M \rtimes P))$ such that $\Theta(\El(Q))=pu^{\ast} \El(P)up$.
	\end{enumerate}
	
\end{theorem}
\begin{proof}
	From the definition of class $\mathscr{S}_{m,n}$ we have $Q=Q_1\times Q_2\times \cdots\times Q_n$, $P=P_1\times P_2\times\cdots\times P_n$, $N=N_1\times N_2\times \cdots\times N_m$ and $M=M_1\times M_2\times\cdots\times M_m$, where $Q_i$'s and $P_i$'s are nontrivial,  icc, torsion free, biexact, weakly amenable, residually finite, property (T) groups for all $i\in\overline{1,n}$ and $N_j$'s and $M_j$'s have property (T) for all $j\in\overline{1,m}$.  Denote $\emm= \mathcal L(M\rtimes P),\Bee=\Theta(\El(Q))$, $\Aa=\Theta(\El(N))$, $\Aa^i_1=\Theta(\El(N_i))$, $\Aa^i_2=\Theta(\El(\widehat N_i))$, $\Bee^i_1=\Theta(\El(Q_i))$ and $\Bee^i_2=\Theta(\El(\widehat Q_i))$. Then we see that  $\Aa_1^i, \Aa^i_2\subset p\emm p$ are commuting property (T) subalgebras such that $\Aa_1\vee \Aa_2=\Aa$. Using the Theorem \ref{controlcommutingpropTsubalgebratrivialf_group}, for every $i,k\in\overline{1,m}$ and we have that either
	\begin{enumerate}
		\item [a)]$\Aa^i_j \prec^s_{\emm} \El(\widehat M_k)$ or,
		
		\item [b)] $\Aa \prec^s_{\emm} \El(P)$
	\end{enumerate}
	for some $j\in\{ 1,2\}$.
	Assume case b) above holds. 
	Then there exists projections $r \in \Aa$, $q \in \El(P)$, a nonzero partial isometry $v \in q \emm r$, and a $\ast$-homomorphism $\psi:r\Aa r \rightarrow \psi( r \Aa r) \subseteq q \El(P)q$ such that $\psi(x)v=vx$ for all $x \in r \Aa r$. 
	Arguing exactly as in the proof of \cite[Theorem 5.5]{CDK19}, we can show that $v \mathcal{QN}_{r \emm r}(r \Aa r)''v^{\ast} \subseteq  q\El(P) q$. 
	
	Now, $\mathcal{QN}_{r \emm r}(r \Aa r)''= r \emm r$, using \cite[Lemma 3.5]{Po03}. Thus, $\emm \prec_{\emm} \El(P)$ and hence $\El(P)$ has finite index in $\emm$ by \cite[Theorem 2.3]{CD18}, which is a contradiction. Hence we must have a).
	
	Assuming case $a)$ we can see that $\Aa^i_j \prec^s_{\emm} \El(\widehat M_k)$ for some $j\in\{1,2\}$. Now if $j=1$ for all $i$ for some $k\in\overline{1,m}$, then we have $\Aa^i_2 \prec^s_{\emm} \El(\widehat M_k)$ for every $i$. By \cite[Corollary~1]{CDHK20} we get that $\Aa=\vee_i\Aa^i_j \prec^s_{\emm} \El(\widehat M_k)$. Since $\El(\widehat M_k)\subseteq \El(M)$, we have that $\Aa\prec_{ \emm}\El(M)$. Since $N\triangleleft N\rtimes Q$, we have $\Aa\prec^s_{ \emm}p\El(M)p$. Now if for every $k\in\overline{1,m}$ there is some $i$ such that $\Aa^i_2\prec_{ \emm}^s\El(\widehat M_k)$, denote $\Aa^i_2=:\Aa_{i_k}$. We now claim that for some $k\neq s$ we have $i_k\neq i_s$. Assume by contradiction that $i_k=i_s$ for all $k,s\in\overline{1,m}$. Then we have $\Aa_{i_k}\prec^s_{ \emm}\El(\widehat M_k)$ for every $k\in\overline{1,m}$. Hence by \cite[Lemma~2.6]{DHI16} we get that $\Aa_{i_k}\prec_{ \emm}^s\El(\underset{k\in\overline{1,m}}{\cap}\widehat M_k)=\mathbb{C}$, which is a contradiction. So we get that for some $k\neq s$ we have $i_k\neq i_s$. Using the fact that $\El(N_i)\subseteq \El(N)=\Aa_{i_k}\cup\Aa_{i_s}$, $\Aa_{i_j}\prec^s_{ \emm}\El(\widehat M_j)$ for $j=\{k,s\}$ together with the fact that $N_i\triangleleft N\rtimes Q$ we get that for all $i\in\overline{1,m}$, we have $\El(N_i)=\Aa^i_1\prec_{ \emm}^sp\El(M)p$. This together with \cite[Corollary~1]{CDHK20} gives us $\Aa\prec_{ \emm}^sp\El(M)p$. 
	
	Now we can argue similar to the proof of \cite[Theorem~4.4]{CDHK20} and conclude that there exists a unitary $v\in\mathcal{U}(p\El(M\rtimes P)p)$ such that $\Theta(\El(N))=vp\El(M)pv^{\ast}$. This proves $1.$ of the theorem. 
	
	We now see that  $\Bee_1, \Bee_2\subset p\emm p$ are commuting property (T) subalgebras such that $\Bee_1\vee \Bee_2=\Bee$. Moreover, we have that $\{\Bee_1\vee \Bee_2\}' \cap p\emm p=\Bee' \cap \Theta(\El(N \rtimes Q)) =\mathbb C \Theta (1)=\mathbb C p$.  Hence by Theorem~\ref{intertwiningthm}, we either have that a) $\Bee \prec^s_{\emm} \El(M)$, or b) $\Bee \prec^s_{\emm} \El(P)$. By part 1. we also know that $\Aa \prec_{\emm}^s \El(M)$. Thus, if a) holds, then \cite[Theorem~2.3]{CDHK20} implies that $p\emm p=\Theta(\El(N\rtimes Q)) \prec_{\emm} \El(M)$. In turn this implies that $Q$ is finite, a contradiction. Hence b) must hold, i.e.\ $\Bee \prec^s_{\emm} \El(P)$. Now we can argue similar to the proof of \cite[Theorem~4.5]{CDHK20} and conclude that there exists a unitary $u\in\mathcal{U}(p\El(M\rtimes P)p)$ such that $\Theta(\El(N))=up\El(P)pu^{\ast}$. This proves $2.$ of the theorem. 
\end{proof}

Next we are ready to derive our main theorem for this section. 

\begin{theorem} \label{mainthm}
	Let $N\rtimes Q, M\rtimes P \in \mathscr S_{m,n}$ with $N=N_1\times N_2\times\cdots\times N_m$ and $M=M_1\times M_2\times\cdots\times M_m$. Let $p\in \El(M\rtimes P)$ be a projection and assume that  $\Theta: \El(N \rtimes Q) \rightarrow p\El(M \rtimes P)p$ is a $\ast$-isomorphism. Then $p=1$ and one can find  $\ast$-isomorphisms, $\Theta_i: \El(N_i) \rightarrow \El(M_i)$, a group isomorphism $\delta: Q \rightarrow P$, a multiplicative character $\eta: Q \rightarrow \mathbb T$, and a unitary $u \in \mathscr U( \El(M \rtimes P))$ such that  for all $ g \in Q$, $x_i \in N_i$ we have that$$ \Theta((x_1 \otimes x_2\otimes \cdots\otimes x_m)u_{g})= \eta(g)u(\Theta_1(x_1) \otimes \Theta_2(x_2)\otimes\cdots\otimes\Theta(x_m))v_{\delta(g)})u^{\ast}. $$
\end{theorem}

\begin{proof}
	Throughout this proof we will denote by $\emm=\El(N\rtimes Q)$. Using Theorem~\ref{core-actinggp}, and replacing $\Theta$ by $\Theta \circ Ad(v)$ if necessary, we may assume that $\Theta(\El(N))= p\El(M)p$. By Theorem~\ref{core-actinggp}, there exists $u \in \mathscr U(\emm)$ such that $ \Theta(\El(Q)) \subseteq u^{\ast}\El(P)u$, where $\emm= \El(M \rtimes P)$. Moreover $\Theta(1)=p$, $upu^{\ast} \in \El(P)$ and also $\Theta(\El(Q))= pu^{\ast}\El(P)up$. Next we denote by $\G= u^{\ast} P u$ and by $\mathcal G =\{\Theta (u_g)\, :\,g\in Q\}$. Notice that $\mathcal G \subseteq \El(\G)$ is a group of unitaries normalizing $\Theta(\El(N))$. Notice that the action $\sigma: P\rar {\rm Aut}( M )$ satisfies all the conditions in the hypothesis of \cite[Theorem~2.5]{CDHK20} and thus using the conclusion of the same theorem we get the following

	$$h_{\G}(\mathcal G)>0.$$
	
	Since $\G$ is isomorphic to the product of $m$ biexact groups, say $\G_1 \times \G_2\times\cdots\times \G_m$, then we get that $C_{\G}(g)= A$, $\widehat \G_i \times A$, for an amenable group $A$ and $i\in\overline{1,m}$. If $C_{\G}(g)= A$ then since  $\mathcal G$ is non-amenable we clearly have $\mathcal G'' \nprec \El(C_\G(g))$.  Next assume $C_{\G}(g)= A\times \widehat \G_i$ and assume by contradiction that $\mathcal G''\prec \El(C_\G(g))$. As $Q=Q_1\times Q_2\times \cdots\times Q_m$ for $Q_i$ property (T) icc group, then $\mathcal G'' =\underset{j\in\overline{1,m}}{\bar\otimes}\Theta (\El(Q_j)) $ is a $\rm II_1$ factor with property (T). Since $\mathcal G '' \prec \El(A\times \widehat \G_i)=\El(A)\bar\otimes \El(\widehat \G)$  and $\El(A)$ is amenable then it follows that $\mathcal G '' \prec \El(\widehat \G)$. However by \cite{OP07} this is impossible as $\El(\G_i)$'s are solid and $\mathcal G''$ is generated by two non-amenable commuting subfactors. Hence we have shown the following:
	\begin{center}
		$\mathcal G '' \nprec\El(C_{\G}(g))$ for $e\neq g\in\G$.
	\end{center}
	Note that we have $\Theta(\El(Q))=\mathcal G''=p\El(\G)p$. Using the fact that $Q$ is icc together with \cite[Proposition~3.4]{CSU13} we get that the representation $Ad(Q)$ on $L^2(\El(Q)\ominus\mathbb C)$ is weak mixing. Hence combining these we get the following:
	\begin{center}
		The unitary representation $\{{\rm Ad}(v)\}_{v \in \mathcal G}$ on $L^2(p \El(\G)p \ominus \mc p)$ is weakly mixing.
	\end{center}  
	All the above results together with \cite[Theorem~2.4]{CDHK20} implies that $p=1$. The arguing exactly same as in the proof of \cite[Theorem~4.6]{CDHK20} we can conclude our theorem. 
\end{proof}

\begin{corollary}\label{trivialfundamentalgroup_S_m,n}
	For any group $G=N\rtimes Q\in\mathscr{S}_{m,n}$ with $n,m\geq 2$, the fundamental group of the corresponding group von Neumann algebras $\El(G)$ is trivial, i.e. $\mathcal F(\El(G))=\{1\}$.
\end{corollary}

We end this section by showing that $\El(\G)$ is prime if $\G\in\mathscr{S}_{m,n}$. This can be deduce from Theorem~\ref{mainthm} together with \cite[Lemma~2.3]{D20}.

\begin{corollary}\label{primeSm,n}
	Let $\G\in\mathscr{S}_{m,n}$ with $n,m\geq 2$. Then $\El(\G)$ is prime.
\end{corollary}

\begin{proof}
	Let $\emm=\El(\G)=\El(N_1\times\cdots\times N_m)\rtimes Q=\mathcal P_1\bten \mathcal P_2$. Since $\mathcal P_1,\mathcal P_2$ are commuting property (T) subfactor of $\emm$ such that $(\mathcal P_1\vee\mathcal P_2)'\cap \emm=\mathbb{C}$. Hence by using Theorem~\ref{intertwiningthm}, we get that 
	\begin{enumerate}
		\item[a)] either $\mathcal P_1\vee\mathcal P_2\prec^s_{\emm}\El(N_1\times\cdots\times N_m)$, or
		\item[b)] $\mathcal P_1\vee\mathcal P_2\prec^s_{\emm}\El(Q)$. 
	\end{enumerate}
	Since $\mathcal P_1\vee \mathcal P_2=\El(\G)$, both cases above leads to a contradiction and hence we get that $\El(\G)$ is prime.  
\end{proof}


\section{Trivial fundamental group arising from Class $\mathscr{S}_{\infty,n}$}

Consider $\emm:=\otimes_{i=1}^{\infty}\El(\G)$, where $\G\in\mathscr{S}_{m,n}$. Then we have  that $\mathcal{F}(\bten_{i=1}^{k}\El(\G))=\{1\}$ for all $k\geq 1$ by Corollary~\ref{trivialfundamentalgroup_S_m,n} and Corollary~\ref{upffrom_s_mn}. However, note that $\mathcal F(\emm)=\mathcal F(\lim_{k\rightarrow\infty} \bten_{i=1}^{k}\El(\G))=\mathbb R_+$ as $\emm$ is a McDuff factor. In this section we are going to consider inductive limits of groups arising from $\mathscr{S}_{m,n}$ by letting $m\rightarrow \infty$. Novelty of this section is that we  show that the group von Neumann algebras associated to this class of groups has trivial fundamental group.

\begin{theorem}\label{slekt-counter}
Let $N\rtimes Q, M\rtimes P \in \mathscr S_{\infty,n}$ with $N=N_1\times N_2\times\cdots$ and $M=M_1\times M_2\times\cdots$. Let $p\in \El(M\rtimes P)$ be a projection and assume that  $\Theta: \El(N \rtimes Q) \rightarrow p\El(M \rtimes P)p$ is a $\ast$-isomorphism. Then $p=1$ and one can find  $\ast$-isomorphisms, $\Theta_i: \El(N_i) \rightarrow \El(M_i)$, a group isomorphism $\delta: Q \rightarrow P$, a multiplicative character $\eta: Q \rightarrow \mathbb T$, and a unitary $u \in \mathscr U( \El(M \rtimes P))$ such that  for all $ g \in Q$, $x_i \in N_i$ we have that$$ \Theta((x_1 \otimes x_2\otimes \cdots\otimes x_m)u_{g})= \eta(g)u(\Theta_1(x_1) \otimes \Theta_2(x_2)\otimes\cdots\otimes\Theta(x_m))v_{\delta(g)})u^{\ast}. $$
 \end{theorem}

For every finite set $F\subseteq \mathbb N$, define $M^j_F:=N^j_F\rtimes Q$ (similarly $M_F$) where $N^j_F:=\times_{i\in F} N^j_i$ for all $j\in\{1,2\}$. Note that $M_F$ has property $(T)$ for every $F$ and $\emm=\underset{F}{\cup} \El(M^j_F)$. Since $\El(N^j_F)$ has property (T), we have 
\begin{align}
	\El(M^1_F)\prec \El(M^2_{F'}) \ for \ some F'.
\end{align}\label{finitesubsetintertwin}
Notice that for every finite set $F$ we have $\emm\subseteq \enn_{\emm}(\El(N^1_F))''$. Hence we get $\El(M^1_F)\prec^s \El(M^2_{F'})$. Without loss of generality we can now assume that $F=\{1,\cdots,k\}$ and $F'=\{ 1,\cdots, l\}$. So we have,
\begin{align}
	\El(N^1_F\rtimes Q^1)\prec^s_{\emm}\El(N^2_{F'}\rtimes Q^2).
\end{align}
For every $i\in\overline{1,k}$, let us denote $\Aa^1_i:=\Theta(\El(\times_{j\in F\setminus\{i\}} N_j))=\Theta(\El(\widehat N_i))$ and $\Aa^2_i:=\Theta(\El(N_i))$. Note that $\Aa^1_i,\Aa^2_i$ are two commuting, property $(T)$, regular subfactor of $p\emm p$ Notice that $A^j_i\subseteq \El(N^1_F\rtimes Q^1)$ and \ref{finitesubsetintertwin} implies that $\Aa^1_i$
Hence there exist projections $p\in \El(M^1_F)$ and $q\in\El(M^2_{F'})$, nonzero partial isometry $v\in\emm$ and a $ast$-isomorphism $\phi:p\El(N^1_F\rtimes Q^1) p\rightarrow \mathcal B:=\phi(p\El(N^1_F\rtimes Q^1) p)\subseteq q \El(N^2_{F'}\rtimes Q^2)q$ such that
\begin{align}
	\phi(x)v=vx\ \textit{for all } x\in p\El(N^1_F\rtimes Q^1) p.
\end{align}
 Note that for every $j\in\overline{1,k}$ we denote $A_i=\phi(\Aa^i_j)$ for $i=1,2$. Now $\Aa_i$'s are commuting, property (T) subfactor of $\El(M_{F'}\rtimes P)$. Hence by Theorem~\ref{controlcommutingpropTsubalgebratrivialf_group}, for every $s\in F'$ we have the following:
 \begin{enumerate}
 	\item[i.] There exists $i\in\overline{1,2}$ such that $\mathcal A_i\prec_{ \emm}\El(\widehat M_s)$;
 	\item[ii.] $\mathcal A_1\vee\mathcal A_2\prec_{ \emm} \El(P)$.
 \end{enumerate} 
Assuming $ii.$ and the fact $N_F\triangleleft N\rtimes Q$, we have $\El(N_F)\prec^s_{ \emm}\El(P)$. Then there exists projections $r \in \El(N_F)$, $q \in \El(P)$, a nonzero partial isometry $v \in q \emm r$, and a $\ast$-homomorphism $\psi:r\El(N_F) r \rightarrow \psi( r \Aa r) \subseteq q \El(P)q$ such that $\psi(x)v=vx$ for all $x \in r \El(N_F) r$. 
Arguing exactly as in the proof of \cite[Theorem 5.5]{CDK19}, we can show that $v \mathcal{QN}_{r \emm r}(r \El(N_F) r)''v^{\ast} \subseteq  q\El(P) q$. 

Now, $\mathcal{QN}_{r \emm r}(r\El(N_F) r)''= r \emm r$, using \cite[Lemma 3.5]{Po03}. Thus, $\emm \prec_{\emm} \El(P)$ and hence $\El(P)$ has finite index in $\emm$ by \cite[Theorem 2.3]{CD18}, which is a contradiction. So we must have $i.$.

Now assuming i. for every $s\in F'$ and for every $j\in \overline{1,k}$ we have either $\El(N_j)\prec^s_{ \emm}\El(\widehat M_s)$ or  $\El(\widehat N_j)\prec^s_{ \emm}\El(\widehat M_s)$. 

If for some $s$ and for all $j$ we have  $\El(N_j)\prec^s_{ \emm}\El(\widehat N_s)$, then by \cite[Corollary~1]{CDHK20} we get that $\El(N_F)=\vee_j \El(N_j)\prec^s_{ \emm}\El(\widehat M_s)\subseteq \El(M_{F'})$. Using normality again we get that $\El(N_F)\prec_{ \emm}^s\El(M_{F'})$. 

If this is not the case then for every $s\in F'$ there is $j_s\in\overline{1,k}$ such that $\El(\widehat N_{j_s})\prec^s_{\emm}\El(\widehat M_s)$. Now we claim that there is $s\neq t$ such that $j_s\neq j_t$. Let us assume by contradiction that $j_s=j_t$ for all $s,t\in F'$. Then we get that $\El(\widehat N_{j_s})\prec^s_{\emm}\El(\widehat M_s)$ for every $s\in F'$. By \cite[Lemma~2.6]{DHI16} this then implies that $\El(\widehat N_{j_s})\prec^s_{\emm}\El(\cap_{s\in F}\widehat N_s )=\mathbb C$ which is a contradiction. Hence our claim is true. Now arguing same as Theorem~\ref{core-actinggp} we conclude that $\El(N_F)\prec^s_{\emm}\El(M_{F'})$. Notice that $\El(M_{F'})\subseteq \El(M)$. Using this together with the fact that $\El(N_F)$ is regular in $\emm$ and \cite[Lemma~2.6]{DHI16} we conclude that for every finite subset $F$ we get that $\El(N_F)\prec_{ \emm}^s\El(M)$. Using \cite[Theorem~2.7]{DHI16} and \cite[Theorem~1.4]{PV12} we conclude that $\El(N)\prec_{ \emm}^s\El(M)$.

Now repeating the above argument other way, we get that $p\El(M)p\prec_{p\emm p} \Aa$. Using exactly same as the last part of \cite[Theorem~4.4]{CDHK20} we conclude that there exists a unitary $v\in \mathcal U(p\El(M\rtimes P)p)$ such that $\Theta(\El(N))=vp\El(M)pv^{\ast}$.

Arguing exactly as in the last part of Theorem~\ref{core-actinggp} we get unitary $u\in\mathcal U(\El(M\rtimes P))$ such that we have $\Theta(\El(Q))=pu^{\ast}\El(P)up$.

 Next we denote by $\G= u^{\ast} P u$ and by $\mathcal G =\{\Theta (u_g)\, :\,g\in Q\}$. Notice that $\mathcal G \subseteq \El(\G)$ is a group of unitaries normalizing $\Theta(\El(N))$. Notice that the action $\sigma: P\rar {\rm Aut}( M )$ satisfies all the conditions in the hypothesis of \cite[Theorem~2.5]{CDHK20} and thus using the conclusion of the same theorem we get the following

$$h_{\G}(\mathcal G)>0.$$

The following statement follows directly from the proof of the second part of the Theorem~\ref{mainthm}:
\begin{center}
	$\mathcal G '' \nprec\El(C_{\G}(g))$ for $e\neq g\in\G$.
\end{center}
Note that we have $\Theta(\El(Q))=\mathcal G''=p\El(\G)p$. Using the fact that $Q$ is icc together with \cite[Proposition~3.4]{CSU13} we get that the representation $Ad(Q)$ on $L^2(\El(Q)\ominus\mathbb C)$ is weak mixing. Hence combining these we get the following:
\begin{center}
	The unitary representation $\{{\rm Ad}(v)\}_{v \in \mathcal G}$ on $L^2(p \El(\G)p \ominus \mc p)$ is weakly mixing.
\end{center}  
All the above results together with \cite[Theorem~2.4]{CDHK20} implies that $p=1$.

\begin{corollary}\label{trivialfundamentalgroupS_infty,n}
	For any group $G=N\rtimes Q\in\mathscr S_{\infty,n}$ for $n\geq 2$, $\mathcal F(\El(G))=\{1\}$.
\end{corollary}

We now mention that Theorem~\ref{slekt-counter} provides examples of infinite families of finite index subgroups $\G_n\leq \G$ in a given non-property (T) group $\G$ such that the corresponding group factors $\El(\G_n)$ and $\El(\G_m)$ are stably non-isomorphic for all $n\neq m$. As $\G_n$'s are measure equivalent this provides new counter examples to D. Shlyakhtenko's question different from the one obtained in \cite{CI09,CdSS15,CDK19}. 

\begin{corollary}
	Let $Q_1,Q_2,\cdots ,Q_n$ be uniform lattices in Sp(n,1) with $n\geq 2$ and let $Q=Q_1\times Q_2\times\cdots\times Q_n$. Also let $\cdots \leq Q_i^s\leq\cdots\leq Q^2_i\leq Q^1_i\leq Q_i$ be an infinite family of finite index subgroups for some $i\in\overline{1,n}$ and denote by $Q_s:=(\times_{j\neq i}Q_j)\times Q^s_i$. Now consider $\G_s:=N\rtimes Q_s\leq N\rtimes Q=:\G\in \mathscr S_{\infty,n}$. Note that this is a finite index inclusion of groups. Consider the family $\mathcal F:=\{\El(\G_s)\}_{s\in I}$. Note that being finite index subgroups we have $\G_n$ measure equivalent to $\G_m$ for any $m\neq n$ but Theorem~\ref{slekt-counter} implies that the family consists of mutually stably non-isomorphic group factors.  
\end{corollary}

\section{Uniques Prime Factorization and Calculation of Fundamental Group}\label{uniqueprimeness} 

In this section we are going to show some unique prime factorization results arising from class $\mathscr S_{m,n}$ and class $\mathscr V$ and as a consequence we obtain some more examples of groups whose group von Neumann algebra has trivial fundamental group. We briefly recall class $\mathscr V$. $Sp(n,1)$ is the set of real points of an algebraic $\mathbb Q$-group. The group of integer points $\La_n:=Sp(n,1)_{\mathbb Z}$ is a lattice in $Sp(n,1)$. Now observe that $Sp(n,1)\ca \mathbb H^{n+1}=\mathbb R^{4(n+1)}$ in a way that $\La_n$ preserves $(\mathbb H_{\mathbb Z})^{n+1}\cong \mathbb Z^{4(n+1)}$. Now consider the natural semidirect product $G_n:=\mathbb Z^{4(n+1)}\rtimes \La_n$ for every $n\geq 2$. Class $\mathscr V$ consists of all $k$-fold products of all such groups for any $k\geq 1$.  

The following theorem follows directly from \cite[Proposition~3.1,Theorem~3.2]{D20}, but we give some details for the convenience of the reader.
\begin{theorem}\label{UPFSV}
	Let $\G\in \mathscr S_{m,n}$ and $\La\in\mathscr V$ be two groups. Then $\El(\G\times\La)$ admits a unique prime factorization.  
\end{theorem}

\begin{proof}
	Let us first assume that $\La\in\mathscr{V}_1$. Let $\La=\Aa\rtimes H$, where $\Aa$ is amenable and $H$ is hyperbolic group. Suppose $\emm=\El(\G\times \La)=\mathcal P_1\bten\mathcal P_2$. Observe that $\emm=\El(\G\times \Aa)\rtimes H$ where $H\ca\El(\G)$ is trivial. For a diffuse, amenable subalgebra $\Bee\subseteq \cp_1$, using \cite{PV12}, we have either
	\begin{enumerate}
		\item[a.] $\Bee\prec_{ \emm}\El(\G\times \Aa)$, or
		\item[b.] $\cp_2\subseteq \enn_{\emm}(\Bee)''\lessdot \El(\G\times \Aa)$. 
	\end{enumerate}
Assuming a. in conjenction with \cite[Corollary~F.14]{BO08} we get that $\cp_1\prec_{ \emm}\El(\G\times \Aa)$. Assuming b. together with the fact that $\cp_2$ has property (T), we get that $\cp_2\prec_{ \emm}\El(\G\times\Aa)$. Without loss of generality we assume that $\cp_1\prec_{ \emm}\El(\G\times \Aa)$. Then following the proof of \cite[Proposition~3.1]{D20}, we get that $\cp_1\prec_{ \emm}\El(\G)$. Then by \cite[Proposition~12]{OP03} and \cite{Ge96} there is $s>0$,  $u\in\mathcal U(\emm)$ and a von Neumann subalgebra $Q\subseteq \cp_2$ such that $u\El(\G)^su^{\ast}=\cp_1\bten Q$. Since $\El(\G)$ is prime by Corollary~\ref{primeSm,n}, we get that $Q$ is finite dimensional. Hence there exist some $t>0$ such that $u\El(\G)^tu^{\ast}=\cp_1$ and $u\El(\La)^{\frac{1}{t}}u^{\ast}=\cp_2$.  

Now the general case when $\La\in\mathscr{V}$ follows from an induction argument as used in \cite[Theorem~3.2]{D20}.
\end{proof}

We now show unique prime factorization result for class $\mathscr{S}_{m,n}$ generalizing the UPF result for class $\mathscr{S}$ from \cite[Theorem~3.4]{D20}.

\begin{theorem}\label{upfS}
	Let $\G_i\in\underset{m,n\geq 2}{\bigcup}\mathscr{S}_{m,n}$ for $i=1,2$. Then $\El(\G_1\times \G_2)$ admits a UPF.
\end{theorem} 

\begin{proof}
	Let $\G_1=(N_1\times\cdots\times N_l)\rtimes Q$,  $\G_2=(M_1\times\cdots\times M_k)\rtimes P$ and $\emm:=\El(\G_1\times \G_2)=\mathcal S_1\bten\mathcal S_2$. Note that it is enough to show that $\mathcal S_i\prec_{ \emm}\El(\G_2)$ for some $i=\{1,2\}$.
	
	Let us denote $G_i=N_i\rtimes Q$, $G_{l+j}=M_j\rtimes P$, $H_i=Q$ and $H_{l+j}=p$ for $i\in\overline{1,l}$ and $j\in\overline{1,k}$. Note that $G_i$ are hyperbolic relative to $H_i$ respectively for all $i\in\overline{1,l+k}$. Let $\tilde{\emm}:=\El(\times_{i=1}^{l+k} G_i)$. Note that $\El(\G_1\times \G_2)\subseteq \tilde{\emm}$ and $\emm'\cap\tilde{\emm}=\mathbb C$. Observe that $\mathcal S_i$ has property (T) for all $i=1,2$ since $\emm $ has property (T). We now have $\mathcal S_1$ and $\mathcal S_2$ two commuting, property (T) subfactor of $\tilde{\emm}$. By \cite[Theorem~5.3]{CDK19} we have for every $j\in\overline{1,l+k}$ either,
	\begin{enumerate}
		\item[a)] $\mathcal S_i\prec_{\tilde \emm}\El(\widehat G_j)$ or,
		\item[b)] $\mathcal S_1\vee\mathcal S_2\prec_{\tilde \emm}\El(\widehat G_j\times H_j)$.
	\end{enumerate}
We first show that b) does not happen. We are only going to show this for $j=1$ and all the other cases are similar. Assume that $\mathcal S_1\vee\mathcal S_2\prec_{\tilde{\emm}}\El(\widehat \G_1\times Q)$. This implies that $\El(\G_1)\prec_{ \emm}\El(\widehat \G_1\times Q)$. Arguing similarly as in the first part of Theorem~\ref{commutationcontrolincommultiplication} together with \cite[Lemma~2.5]{CDK19} we get that $\El(\G_1)\prec_{\El(\G_1)}\El(\widehat N_1\rtimes Q)$. Since $[\G_1:\widehat N_1\rtimes Q]=\infty$, this is a contradiction by \cite[Proposition~2.3]{CD18}. 

Hence, we have a), i.e. for every $j\in\overline{1,l+k}$ there is $i_j\in\{1,2\}$ such that $\mathcal S_{i_j}\prec_{\tilde \emm}\El(\widehat G_j)$. Letting $j=1$ we get $\mathcal S_{i_1}\prec_{\tilde \emm}\El(\widehat G_1)$. Notice that $\mathcal S_{i_j}\subseteq \El(\G_1\times \G_2)$. Hence by applying \cite[Lemma~2.3]{CDK19} we get that $\mathcal S_{i_1}\prec_{\tilde \emm} \El((\G_1\times\G_2)\cap h(\widehat G_1)h^{-1})$ for some $h\in \times_{i=1}^{l+k} G_i$. We can see that $(\G_1\times\G_2)\cap h(\widehat G_1)h^{-1}=(1\times\widehat N_1)\times\G_2$. SO we have $\mathcal S_{i_1}\prec_{\tilde \emm}\El(\widehat N_1\times\G_2)=\El(\widehat N_1)\bten\El(\G_2)$. Running the similar argument for $j=2$ we obtain that $\mathcal S_{i_2}\prec_{\tilde \emm}\El(\widehat N_2\times\G_2)=\El(\widehat N_2)\bten\El(\G_2)$. We now show that $i_1=i_2$. Assume by contradiction that we have $i_1\neq i_2$. Then without loss of generality we have that, $\mathcal S_1\prec_{\tilde \emm}\El(\widehat N_1)\bten\El(\G_2)$ and $\mathcal S_2\prec_{\tilde \emm}\El(\widehat N_2)\bten\El(\G_2)$. Hence we have $\mathcal S_i\prec_{\tilde{\emm}}\El(N)\bten\El(\G_2)$. Note that $\enn_{\tilde{\emm}}(\mathcal S_i)'\cap\tilde{\emm}\subseteq (\mathcal S_1\vee\mathcal S_2)'\cap\tilde{\emm}=\emm'\cap\tilde{\emm}=\mathbb C. $ As $\mathcal S_1,\mathcal S_2$ are commuting factors and $\El(N)\bten\El(\G_2)$ is regular in $\tilde{\emm}$, by \cite[Lemma~2.6]{Is16} we get that $\mathcal S_1\vee\mathcal S_2=\emm\prec_{\tilde{\emm}}\El(N)\bten\El(\G_2)$. We can now apply \cite[Lemma~2.4]{CDK19} to conclude that $\El(\G_1)\prec_{\El(\G_1)}\El(N)$. This is a contradiction to \cite[Proposition~2.3]{CD18} as $[\G_1:N]=\infty$. Similarly we can show that for every $j\in\overline{1,l}$ we have $i_l=i_1$. Hence we have some $i\in\{1,2\}$ such that $\mathcal S_i\prec_{\tilde{\emm}}\El(\widehat N_j)\bten\El(\G_2)$ for all $j\in\overline{1,l}$. Since we have $\enn_{\tilde{\emm}}(\mathcal S_i)'\cap\tilde{\emm=\mathbb C}$, using \cite[Lemma~2.4(2)]{DHI16} we can conclude that $\mathcal S_i\prec^s_{\tilde{\emm}}\El(\widehat N_j)\bten\El(\G_2)$ for all $j\in\overline{1,l}$. Using \cite[Lemma~2.8(2)]{DHI16} we further conclude that $\mathcal S_i\prec^s_{\tilde{\emm}} \underset{j\in\overline{1,l}}{\bigcap}(\El(\widehat N_j)\bten\El(\G_2))=\El(\G_2)$. Finally by \cite[Lemma~2.6]{CDK19} we conclude that $\mathcal S_i\prec_{ \emm}\El(\G_2)$. 
\end{proof} 

\begin{corollary}\label{upffrom_s_mn}
	Let $\G_1,\G_2\in \underset{m,n\geq 2}{\bigcup} \mathscr{S}_{m,n}\cup \mathscr{V}$. Then $\mathcal F(\El(\G_1\times\G_2))=\{1\}$.
\end{corollary}

\begin{proof}
	Note that when $\G_1,\G_2\in\mathscr{V}$, this case is already covered by \cite[Theorem~5.2]{CDHK20}. Hence the only new case is when $\G_1\in \underset{m,n\geq 2}{\bigcup} \mathscr{S}_{m,n}\cup \mathscr{V}$ and $\G_2\in\underset{m,n\geq 2}{\bigcup} \mathscr{S}_{m,n}$. Let $t\in\mathcal{F}(\El(\G_1\times\G_2))$. Then $\El(\G_1)\bten\El(\G_2)=\El(\G_1)^t\bten\El(\G_2)$. Now by Theorem~\ref{upfS} and Theorem~\ref{UPFSV} we get either 
	\begin{enumerate}
		\item[i.] $\El(\G_1)\simeq\El(\G_1)^{st}$ and $\El(\G_2)\simeq\El(\G_2)^{\frac{1}{s}}$ or,
		\item[ii.] $\El(\G_1)\simeq\El(\G_2)^{s}$ and $\El(\G_2)\simeq\El(\G_1)^{\frac{t}{s}}$.  
	\end{enumerate}
If i. holds, then we get $st=1$ and $1/s=1$ from the fact that $\mathcal (\El(\G_i))=\{1\}$ for $i=1,2$. This then implies that $t=1$. If ii. holds then by Theorem~\ref{mainthm} we get that $s=1$ and $t/s=1$, which implies that $t=1$.
\end{proof}

\section{Cartan-rigidity for von Neumann algebras of groups in Class $\mathscr{S}_{m,n}$ and products of relative hyperbolic groups}
In this last section we classify the Cartan subalgebras in II$_1$ factors associated with the groups in class $\mathscr S_{m,n}$ and products of groups that are hyperbolic relative to a collection of residually finite subgroups, and their free ergodic pmp actions on probability spaces (see Theorem~\ref{uniquecartan1}, Theorem~\ref{uniquecartantheorem}, and Corollary~\ref{cor:nocartan}). Our proofs rely in an essential way on the methods introduced in \cite{PV12}, \cite{CIK13} and \cite{CDK19} as well as on the group theoretic Dehn filling of relative hyperbolic groups. For convenience we include detailed proofs.

We record the following intertwining lemma that follows from \cite[Theorem~7.1]{CDK19}. We give all the details for the readers' convenience.
\begin{theorem}\label{uniquecartantheorem}
	Let $Q=Q_1\times Q_2\times \cdots\times Q_m$ for some integer $m\geq 2$ where $Q_i$ are residually finite groups. For every $i\in\overline{1,n}$ let $\G_i=N_i\rtimes_{\sg_i}Q\in\mathcal Rip(Q)$ and denote by $\G=(N_1\times N_2\times\cdots\times N_n)\rtimes_{\sg}Q$ the semidirect product associated with the diagonal action $\sg=(\sg_1,\sg_2,\cdots,\sg_n):Q\rightarrow Aut(N_1\times N_2\times\cdots\times N_n)$ for $n\geq 2$. Then the following hold:

	Let $\mathcal P$ be a von Neumann algebra together with an action $\G\curvearrowright \mathcal P$ and denote $\mathcal P\rtimes \G$ by $\emm$. Let $p\in\emm$ be a projection and let $\mathcal A\subseteq p\emm p$ be a masa whose normalizer $\mathscr N_{p\emm p}(A)''\subseteq p\emm p$ has finite index. Then $\mathcal A\prec_{ \emm} \mathcal P$.

\end{theorem}

\begin{proof}
	Since $\Gamma_i =N_i\rtimes Q$ is hyperbolic relative to a residually finite group $Q$, then by \cite[Corollary~5.1]{CIK13} there exist a non-elementary hyperbolic group $H_i$, a subset $T_i\subseteq N_i$ with $|T_i|\geq 2$ and a normal subgroup $R_i\lhd Q$ of finite index such that we have a short exact sequence $$1\rightarrow \ast_{t\in T_i}R_i^t\hookrightarrow \Gamma_i\overset{\varepsilon_i}{\twoheadrightarrow} H_i\rightarrow 1. $$ 
	In particular there are infinite groups $K_1,K_2$ so that $\ast_{t\in T_i}R_i^t=K_1\ast K_2$.
	
	Denote by $\pi_i: \Gamma\twoheadrightarrow \Gamma_i $ the canonical projection given by $\pi_i((n_1,n_2,\cdots,n_k) q)=n_i q$ for all $k\geq i$, for all $(n_1,n_2,\cdots,n_k)q\in \Gamma^k\leq \G $. Then for every $i\in\mathbb N$ consider the epimorphism  $\rho_i= \varepsilon_i\circ \pi_i: \Gamma \rightarrow H_i$.  
	Following \cite[Section 3]{CIK13} consider the $\ast$-embedding $\Delta^{\rho_i}: \emm \rightarrow \emm \bar\otimes \El(H_i):=\tilde \emm_i$ given by $\Delta^{\rho_i}(xu_g )= x u_g \otimes v_{\rho_i(g)}$ for all $x\in \emm$, $g\in \Gamma$. Here $(u_g)_{g\in \Gamma}$ and $(v_{h})_{h\in H_i}$ are the canonical group unitaries in $\mathcal P\rtimes \Gamma$ and $\El(H_i)$, respectively.  As $\mathcal A$ is amenable, \cite[Theorem 1.4]{PV12} implies either a) $\Delta^{\rho_i}(\Aa)\prec_{\tilde \emm_i} \emm\otimes 1$ or b) the normalizer  $\Delta^{\rho_i}(\mathscr N_{p\emm p}(\Aa)'')$ is amenable relative to $\emm \otimes 1$ inside $\tilde \emm_i$. Assume b) holds. As $\mathscr N_{p\emm p}(\Aa)''\subseteq p \emm p$ has finite index it follows that $\Delta^{\rho_i}(p\mathcal M p)$ is amenable relative to $\emm \otimes 1$ inside $\tilde \emm_i$. However, using \cite[Proposition 3.5]{CIK13} this further entails that $H_i$ is amenable, a contradiction. Thus a) must hold and using \cite[Proposition 3.4]{CIK13} we get that $\Aa \prec_\emm \mathcal P\rtimes \ker(\rho_i)$. Let $\mathcal N=\mathcal P \rtimes \ker(\rho_i)$ and using \cite[Proposition 3.6]{CIK13} we can find a projection $0\neq q\in \mathcal N$, a masa $\mathcal B\subset q \mathcal N q$ with $\mathcal Q:=\mathscr N_{q \mathcal N q}(\mathcal B)''\subseteq q\mathcal N q$ has finite index. In addition one can find  projections $0\neq p_0 \in \Aa$, $0\neq q'_0\in \mathcal B'\cap p\emm p$ and a unitary $u\in \emm$ such that $u(\Aa p_0 )u^*=\mathcal B p_0$.
	
	Note that the restriction homomorphism $\pi_i: \ker(\rho_i)\rightarrow K_1\ast K_2$ is an epimorphism with $\ker(\pi_i)= N_{\hat i}$. Consider the $\ast$-embedding  $\Delta^{\pi_i}: \mathcal N \rightarrow \mathcal N \bar\otimes \El(K_1\ast K_2)$ given by given by $\Delta^{\pi_i}(xu_g )= x u_g \otimes v_{\pi_i(g)}$ for all $x\in \mathcal P$, $g\in \ker(\rho_i)$. Denote by $\tilde {\mathcal N}_i:= \mathcal N \bar\otimes \El(\ker(\rho_i))$. Also fix $0\neq z\in \mathcal Z (\mathcal Q'\cap q\mathcal N q)$. Since $\Delta^{\pi_i}(\mathcal B z)\subset \mathcal N \bar \otimes \El(K_1\ast K_2)$ is amenable then using $\cite{Io12,Va13}$ one of the following must hold: 
	\begin{enumerate}
		\item[c)] $\Delta^{\pi_i}(\mathcal Q z)$ is amenable relative to $\mathcal N \otimes 1$ inside $\tilde {\mathcal N}_i$;
		\item[d)] $ \Delta^{\pi_i}(\mathcal Q z)  \prec_{\tilde {\mathcal N}_i} \mathcal N \bar\otimes \El(K_j)$ for some $j=1,2$;
		\item[e)] $\Delta^{\pi_i}(\mathcal B z) \prec_{\tilde {\mathcal N}_i} \mathcal N \otimes 1$.
	\end{enumerate}
	
	Assume c) holds. As $\mathcal Q\subseteq q\mathcal N q$ is finite index so is $\mathcal Q z\subseteq z\mathcal N z$ and \cite[Lemma 2.4]{CIK13} implies that $z\mathcal N z \prec^s \mathcal Q z$ and using \cite[Proposition 2.3 (3)]{OP07} we get that $\Delta^{\pi_i}(z\mathcal N z)$ is amenable relative $\mathcal N \otimes 1$ inside $\tilde {\mathcal N}_i$. Thus \cite[Proposition 3.5]{CIK13} implies that $K_1\ast K_2$ is amenable, a contradiction.
	Assume d) holds. By \cite[Proposition 3.4]{CIK13} we have that $ \mathcal Q z  \prec \mathcal P \rtimes (\pi_i )^{-1}(K_j)$ and using \cite[Lemma 2.4 (3)]{DHI16} one can find a projection $0\neq r\in \mathcal Z(\mathcal Q z'\cap z\mathcal N z)$ such that $\mathcal Q r \prec^s \mathcal P \rtimes  (\pi_i )^{-1}(K_j)$. Since $\mathcal Q z \subseteq z\mathcal N z$ is finite index then so is $\mathcal Q r \subseteq r \mathcal N r$ and thus $r\mathcal N r \prec_{\mathcal N} \mathcal Q r$. Therefore using \cite[Lemma 2.4(1)]{DHI16} (or \cite[Remark 3.7]{Va07}) we conclude that $\mathcal N\prec  \mathcal P \rtimes (\pi_i )^{-1}(K_j)$. However this implies that $\pi^{-1}(K_j) \leqslant \ker(\rho_i)$ is finite index, a contradiction. Hence e) must hold and using \cite[Proposition 3.4]{CIK13} we further get that $\mathcal B z \prec_{\mathcal N} \mathcal P\rtimes N_{\hat i}$. Since this holds for all $z$ we conclude that $\mathcal B \prec^s_{\mathcal N} \mathcal P\rtimes N_{\hat i}$. This combined with the prior paragraph clearly implies that $\Aa \prec \mathcal P \rtimes N_{\hat i}$. 
	
	Using \cite[Lemma~2.8(2)]{DHI16} we get the conclusion that $\Aa \prec_\emm \mathcal P$.
\end{proof}

We now record the following intertwining theorem for products of groups that are hyperbolic relative to a family of  residually finite subgroups. Note that the following theorem can be deduced directly from \cite{CIK13}.

\begin{theorem}\label{uniquecartan1} Let $\G=\G_1\times\G_2\times\cdots\times \G_n$, where $\G_i$ are hyperbolic relative to a family of residually finite subgroups for all $i\in\{1,n\}$. Let $\mathcal P$ be a von Neumann algebra together with an action $\Gamma\ca \mathcal P$ and denote by $\mathcal M = \mathcal P \rtimes \Gamma$. Let $p\in \emm$ be a projection and let $\mathcal A\subset p\emm p$ be a masa whose normalizer $\mathscr N_{p\emm p}(\mathcal A)''\subseteq p \emm p$ has finite index.  Then  $\mathcal A \prec_{\mathcal M} \mathcal P$.     
\end{theorem}

\begin{proof} We are going to prove the statement only for the case when $\G_i$ is hyperbolic to a residually finite subgroup $Q_i$ for every $i\in\{1,n\}.$ General case follows exactly same way and does not hide any technicality. Since $\Gamma_i$ is hyperbolic relative to a residually finite group $Q_i$, then by Theorem \cite[Corollary~5.1]{CIK13} there exist a non-elementary hyperbolic group $H_i$, a subset $T_i\subseteq N_i$ with $|T_i|\geq 2$ and a normal subgroup $R_i\lhd Q_i$ of finite index such that we have a short exact sequence $$1\rightarrow \ast_{t\in T_i}R_i^t\hookrightarrow \Gamma_i\overset{\varepsilon_i}{\twoheadrightarrow} H_i\rightarrow 1. $$ 
	In particular there are infinite groups $K_1,K_2$ so that $\ast_{t\in T_i}R_i^t=K_1\ast K_2$.
	
	Denote by $\pi_i: \Gamma\twoheadrightarrow \Gamma_i $ the canonical projection given by $\pi_i((g_1,g_2,\cdots, g_n))=g_i $, for all $(g_1,g_2,\cdots,g_n)\in \Gamma $. Then for every $i=\{1,n\}$ consider the epimorphism  $\rho_i= \varepsilon_i\circ \pi_i: \Gamma \rightarrow H_i$.  
	Following \cite[Section 3]{CIK13} consider the $\ast$-embedding $\Delta^{\rho_i}: \emm \rightarrow \emm \bar\otimes \El(H_i):=\tilde \emm_i$ given by $\Delta^{\rho_i}(xu_g )= x u_g \otimes v_{\rho_i(g)}$ for all $x\in \emm$, $g\in \Gamma$. Here $(u_g)_{g\in \Gamma}$ and $(v_{h})_{h\in H_i}$ are the canonical group unitaries in $\cp\rtimes \Gamma$ and $\El(H_i)$, respectively.  As $\mathcal A$ is amenable, \cite[Theorem 1.4]{PV12} implies either a) $\Delta^{\rho_i}(\mathcal A)\prec_{\tilde \emm_i} \emm\otimes 1$ or b) the normalizer  $\Delta^{\rho_i}(\mathscr N_{p\emm p}(\mathcal A)'')\lessdot_{\tilde{\emm}_i}\emm \otimes 1$. Assume b) holds. As $\mathscr N_{p\emm p}(\mathcal A)''\subseteq p \emm p$ has finite index it follows that $\Delta^{\rho_i}(p\mathcal M p)$ is amenable relative to $\emm \otimes 1$ inside $\tilde \emm_i$. However, using \cite[Proposition 3.5]{CIK13} this further entails that $H_i$ is amenable, a contradiction. Thus a) must hold and using \cite[Proposition 3.4]{CIK13} we get that $\mathcal A \prec_\emm \cp\rtimes \ker(\rho_i)$. Let $\mathcal N=\cp \rtimes \ker(\rho_i)$ and using \cite[Proposition 3.6]{CIK13} we can find a projection $0\neq q\in \mathcal N$, a masa $\mathcal B\subset q \mathcal N q$ with $\mathcal Q:=\mathscr N_{q \mathcal N q}(\mathcal B)''\subseteq q\mathcal N q$ has finite index. In addition one can find  projections $0\neq p_0 \in \mathcal A$, $0\neq q'_0\in \mathcal B'\cap p\emm p$ and a unitary $u\in \emm$ such that $u(\mathcal A p_0 )u^*=\mathcal B p_0$.
	To this end observe the restriction homomorphism $\pi_i: \ker(\rho_i)\rightarrow K_1\ast K_2$ is an epimorphism with $\ker(\pi_i)= \G_{\hat i}$. As before, consider the $\ast$-embedding  $\Delta^{\pi_i}: \mathcal N \rightarrow \mathcal N \bar\otimes \El(K_1\ast K_2)$ given by given by $\Delta^{\pi_i}(xu_g )= x u_g \otimes v_{\pi_i(g)}$ for all $x\in \cp$, $g\in \ker(\rho_i)$. Denote by $\tilde {\mathcal N}_i:= \mathcal N \bar\otimes \El(\ker(\rho_i))$. Also fix $0\neq z\in \mathcal Z (\mathcal Q'\cap q\mathcal N q)$. Since $\Delta^{\pi_i}(\mathcal B z)\subset \mathcal N \bar \otimes \El(K_1\ast K_2)$ is amenable then using \cite{Io12,Va13} one of the following must hold: c) $\Delta^{\pi_i}(\mathcal Q z)$ is amenable relative to $\mathcal N \otimes 1$ inside $\tilde {\mathcal N}_i$; d) $ \Delta^{\pi_i}(\mathcal Q z)  \prec_{\tilde {\mathcal N}_i} \mathcal N \bar\otimes \El(K_j)$ for some $j=1,2$; e) $\Delta^{\pi_i}(\mathcal B z) \prec_{\tilde {\mathcal N}_i} \mathcal N \otimes 1$.
	
	Assume c) holds. As $\mathcal Q\subseteq q\mathcal N q$ is finite index so is $\mathcal Q z\subseteq z\mathcal N z$ and \cite[Lemma 2.4]{CIK13} implies that $z\mathcal N z \prec^s \mathcal Q z$ and using \cite[Proposition 2.3 (3)]{OP07} we get that $\Delta^{\pi_i}(z\mathcal N z)$ is amenable relative $\mathcal N \otimes 1$ inside $\tilde {\mathcal N}_i$. Thus \cite[Proposition 3.5]{CIK13} implies that $K_1\ast K_2$ is amenable, a contradiction.
	Assume d) holds. By \cite[Proposition 3.4]{CIK13} we have that $ \mathcal Q z  \prec \mathcal P \rtimes (\pi_i )^{-1}(K_j)$ and using \cite[Lemma 2.4 (3)]{DHI16} one can find a projection $0\neq r\in \mathcal Z(\mathcal Q z'\cap z\mathcal N z)$ such that $\mathcal Q r \prec^s \cp \rtimes  (\pi_i )^{-1}(K_j)$. Since $\mathcal Q z \subseteq z\mathcal N z$ is finite index then so is $\mathcal Q r \subseteq r \mathcal N r$ and thus $r\mathcal N r \prec_{\mathcal N} \mathcal Q r$. Therefore using \cite[Lemma 2.4(1)]{DHI16} (or \cite[Remark 3.7]{Va07}) we conclude that $\mathcal N\prec  \mathcal P \rtimes (\pi_i )^{-1}(K_j)$. However this implies that $\pi^{-1}(K_j) \leqslant \ker(\rho_i)$ is finite index, a contradiction. Hence e) must hold and using \cite[Proposition 3.4]{CIK13} we further get that $\mathcal B z \prec_{\mathcal N} \cp\rtimes \G_{\hat i}$. Since this holds for all $z$ we conclude that $\mathcal B \prec^s_{\mathcal N} \cp\rtimes \G_{\hat i}$. This combined with the prior paragraph clearly implies that $\mathcal A \prec \cp \rtimes \G_{\hat i}$. 
	
	Since all the arguments above still work and the same conclusion holds if one replaces $\mathcal A$ by $\mathcal A a$ for any projection $0\neq a\in \mathcal A$ one actually has $\mathcal A \prec^s_\emm \cp \rtimes \G_{\hat i}$. Since this holds for every $i\in \{1,2,\cdots,n\}$, using \cite[Lemma 2.8(2)]{DHI16} one concludes that $\mathcal A \prec_\emm \cp\bten (\cap_{i=1}^n\G_{\hat i})=\cp$, as desired. \end{proof}

We get the following Cartan-rigidity results as a corollary for the classes of groups described in the above Theorem and hence in particular for class $\mathscr{S}_{m,n}$ and  class $\mathscr{S}_{\infty,n}$.
\begin{corollary} \label{cor:nocartan} Let $\Gamma$ be a group as in Theorem~\ref{uniquecartan1} and in  Theorem~\ref{uniquecartantheorem}. let $\Gamma \curvearrowright X$ be a free ergodic pmp action on a probability space. Then the following hold:
	\begin{enumerate}
		\item The crossed product $L^\infty(X)\rtimes \Gamma$ has unique Cartan subalgebra;
		\item The group von Neumann algebra $\El(\Gamma)$ has no Cartan subalgebra.
	\end{enumerate}
	
\end{corollary}

\begin{proof} 1. Let $\mathcal A \subset L^\infty(X)\rtimes \Gamma=:\mathcal M$ be a Cartan subalgebra. By Theorem~\ref{uniquecartan1} and Theorem \ref{uniquecartantheorem} we have that $\mathcal A \prec_{\mathcal M}L^\infty(X)$ and since $L^\infty(X)\subseteq \mathcal M$ is Cartan then \cite[Theorem]{Po01}  gives the conclusion.
	2. If $\mathcal A \subset \El(\Gamma)$ is a Cartan subalgebra then Theorem \ref{uniquecartantheorem} implies that $\mathcal A \prec \mathbb C 1$ which contradicts that $\mathcal A$ is diffuse.\end{proof}

\begin{remark}
	By combining Theorem~\ref{uniquecartan1} and Theorem~\ref{uniquecartantheorem} one can show that the Corollary~\ref{cor:nocartan} holds for products of groups that are considered in the Corollary. The proof is essentially same and hence we leave it to the reader. 
\end{remark}

\section*{Acknowledgments}
The authors would like to thank Brent Nelson for the comment to consider more than 2 acting groups in Rips construction while the second author was giving a talk at MSU. The authors are grateful to Ionut Chifan for helpful comments and suggestions.

\noindent
\textsc{Department of Mathematics, University of California Riverside, Skye Hall 219, Riverside, CA 92521, U.S.A.}\\
\email {sdas@ucr.edu} \\
\textsc{Department of Mathematics, The University of Iowa, 14 MacLean Hall, Iowa City, IA 52242, U.S.A.}\\
\email {krishnendu-khan@uiowa.edu} \\

\end{document}